\documentclass[12pt,a4paper,reqno]{amsart}

\usepackage{amssymb}
\usepackage{amscd}
\usepackage[pdftex,pdfpagelabels]{hyperref}
\usepackage{enumerate}
\usepackage{comment}
\usepackage{graphicx}
\usepackage{siunitx}
\usepackage{tikz-cd}
\usepackage{stix}
\usepackage{bm}
\DeclareMathAlphabet\mathbfcal{LS2}{stixcal}{b}{n}
\numberwithin{equation}{section}

\usepackage{mathtools}
\usepackage[tableposition=top]{caption}
\usepackage{booktabs,dcolumn}

\DeclareFontFamily{OT1}{rsfs}{}
\DeclareFontShape{OT1}{rsfs}{n}{it}{<-> rsfs10}{}
\DeclareMathAlphabet{\mathscr}{OT1}{rsfs}{n}{it}

\theoremstyle{plain}

\newtheorem{theorem}{Theorem}[section]

\newtheorem{lemma}[theorem]{Lemma}
\newtheorem{corollary}[theorem]{Corollary}

\theoremstyle{definition}

\newtheorem{remark}[theorem]{Remark}

\renewcommand\P{\mathbb{P}}
\newcommand\E{\mathbb{E}}
\newcommand\Var{\mathrm{Var}}
\newcommand\Cov{\mathrm{Cov}}
\newcommand\R{\mathbb{R}}
\newcommand\Z{\mathbb{Z}}

\newcommand\C{\mathbb{C}}

\renewcommand\Re{{\operatorname{Re}}}

\newcommand\tr{{\operatorname{tr}}}
\newcommand\rhosc{{\rho_{\operatorname{sc}}}}
\newcommand\eps{\varepsilon}

\begin{document}

\title[Minors of GUE at fixed index]{On the distribution of eigenvalues of GUE and its minors at fixed index}

\author{Terence Tao}
\address{UCLA Department of Mathematics, Los Angeles, CA 90095-1555.}
\email{tao@math.ucla.edu}



\begin{abstract}  We obtain bounds on the distribution of normalized gaps of eigenvalues of $N \times N$ GUE matrix in the bulk, that do not lose logarithmic factors of $N$ in the limit $N \to \infty$.  As an application, we obtain fixed index universality results for the GUE minor process, which in turn are useful for establishing limiting results for random hives with GUE boundary data.
\end{abstract}

\maketitle

\section{Introduction}

The purpose of this paper is to establish some distributional estimates on the normalized gaps $g_i^{(M)}$ of the GUE minor process in the bulk.  These estimates will be used in a forthcoming paper of Narayanan \cite{narayanan-limit} to establish limit theorems for random hives with GUE boundary data.

\subsection{Eigenvalues of GUE}

Let $N$ be a large natural number.
Define a \emph{GUE random matrix $H$} to be a random Hermitian $N \times N$ matrix with probability density function proportional to $e^{-\tr H^2}$, so that the off-diagonal entries are complex gaussians $N(0,1/4)_\C$ with variance $1/4$ (and the diagonal entries are real gaussians $N(0,1/2)_\R$ with variance $1/2$); this normalization is chosen to be consistent with that in \cite{gustavsson}, although we will often then divide $H$ by $\sqrt{2N}$ to place the spectrum largely inside the unit interval $[-1,1]$.  The eigenvalues of this matrix are almost surely\footnote{All assertions to this paper are understood to only hold outside of an exceptional event of probability zero.} simple, and can be viewed either as a coupled system of $N$ real random variables
$$ \lambda_1 < \dots < \lambda_N$$
or as an $N$-element point process $\Sigma \coloneqq \{\lambda_1,\dots,\lambda_N\}$, which we will refer to as the \emph{GUE eigenvalue process}.  One can study these eigenvalues either\footnote{There is also the \emph{averaged index} or \emph{averaged energy} setting, in which one is permitted to average in either the index parameter $i$ or the energy parameter $x$.  The two averaged settings are largely equivalent to each other (thanks to the eigenvalue rigidity results discussed below), and are of course easier to obtain results for than the fixed index or fixed energy setting; but it can be non-trivial to pass from averaged results back to fixed index or energy results.} in the \emph{fixed index} setting in which the index $i$ is a fixed function of $N$, or in the \emph{fixed energy} setting in which we inspect the spectrum near $x \sqrt{2N}$ for some fixed energy $x \in [-1,1]$ that will typically be independent of $N$ (and in the \emph{bulk region} $[-1+\delta,1-\delta]$ for some fixed $\delta>0$).

The fixed energy theory of the GUE eigenvalue process is well understood by the theory of determinantal processes.  In particular, for any continuous symmetric compactly supported function $F: \R^k \to \R$, we have the identity
\begin{equation}\label{det}
  \E \sum_{i_1 < \dots < i_k} F(\lambda_{i_1},\dots,\lambda_{i_k}) = \frac{1}{k!} \int_{\R^k} F(x_1,\dots,x_k) \det( K_N(x_i,x_j))_{1 \leq i,j \leq k}\ dx_1 \dots dx_k
\end{equation}
where the kernel $K_N$ can be given by the Christoffel--Darboux formula
$$ K(x,y) = \sqrt{\frac{N}{2}} \frac{h_N(x) h_{N-1}(y) - h_{N-1}(x) h_N(y)}{x-y} e^{-(x^2+y^2)/2}$$
for $x \neq y$ and
$$ K(x,y) = (N h_N(x)^2 - \sqrt{N(N+1)} h_{N-1}(x) h_{N+1}(x)) e^{-x^2}$$
for $x=y$, where the Hermite polynomials
$$ h_N(x) \coloneqq \frac{(-1)^N}{\pi^{1/4} 2^{N/2} \sqrt{N!}} e^{x^2} \frac{d^N}{dx^N} e^{-x^2}.$$
$h_0, h_1, \dots$ are the orthonormal polynomials of $e^{-x^2}\ dx$; see e.g., \cite{mehta}.  As typical applications of this identity, the mean
and variance of the number of eigenvalues $\# (\Sigma \cap I)$ in an interval $I$ can be computed as
$$ \E \# (\Sigma \cap I) = \int_I K_N(x,x)\ dx$$
and
$$ \Var \# (\Sigma \cap I) = \int_I K_N(x,x)\ dx - \int_I \int_I K_N(x,y)^2\ dx dy.$$
Furthermore, $\# (\Sigma \cap I)$ has the distribution of a sum of independent Bernoulli variables, with means given by the eigenvalues of the kernel $K_N$ restricted to $I$; see \eqref{sigma-card} below. Our interest will however be in the fixed index regime, which as we shall see can also be studied from the theory of determinantal processes, but with additional effort.

Many properties of the GUE eigenvalue process are known.  For instance, the Wigner semicircular law asserts that the normalized empirical spectral distribution
$$ \frac{1}{N} \sum_{i=1}^N \delta_{\lambda_i/\sqrt{2N}} = \frac{1}{N} \sum_{\lambda \in \Sigma} \delta_\lambda$$
converges in probability (in the vague topology) to the semicircular distribution $\rhosc(x)\ dx$ given by
$$ \rhosc(x) \coloneqq \frac{2}{\pi} (1-x^2)_+^{1/2}.$$
Equivalently, for any interval $I \subset \R$, the random variable
$$ \frac{1}{N} \# \{ 1 \leq i \leq n: \lambda_i/\sqrt{2N} \in I \} = \frac{1}{N} \# (\Sigma \cap \sqrt{2N} I)$$
converges in probability to $\int_I \rhosc(x)\ dx$.  Another equivalent form of this law is that if $1 \leq i \leq N$, then $\lambda_i/\sqrt{2N}-\gamma_{i/N}$ converges in probability to zero, where the \emph{normalized classical location} $\gamma_{i/N}$ is the unique element of $[-1,1]$ such that
\begin{equation}\label{gammai-def}
  \int_{-\infty}^{\gamma_{i/N}} \rhosc(x)\ dx = \frac{i}{N}.
\end{equation}
Informally, this asserts the heuristic
\begin{equation}\label{lambdai}
\lambda_i \approx \sqrt{2N} \gamma_{i/N} + o(\sqrt{2N})
\end{equation}
in some probabilistic sense.  The phenomenon of \emph{eigenvalue rigidity} ensures that the error term here can be strengthened significantly, as we shall shortly discuss.

Now let us restrict attention to the \emph{bulk region}
\begin{equation}\label{bulk-def}
  \delta n \leq i \leq (1-\delta)n
\end{equation}
for some fixed $\delta>0$; by the semicircular law, this is essentially equivalent to restricting the spectrum $\Sigma$ to $[(-1+\delta')\sqrt{2N}, (1-\delta')\sqrt{2N}]$ for a slightly different constant $\delta'>0$.  In this\footnote{Analogous results at the edge of the spectrum are also established in \cite{gustavsson}.} region, a well known central limit theorem of Gustavsson \cite[Theorem 1.1]{gustavsson} asserts that the normalized eigenvalue
$$ \frac{\lambda_i - \gamma_{i/N} \sqrt{2N}}{\left(\frac{\log N}{4(1-\gamma_{i/N}^2)N}\right)^{1/2}}$$
converges in distribution to the usual gaussian distribution $N(0,1)_\R$.  Informally, this implies that the heuristic \eqref{lambdai} can in fact be strengthened in the bulk region \eqref{bulk-def} to
\begin{equation}\label{lambdai-improv}
  \lambda_i \approx \sqrt{2N} \gamma_{i/N} + O_\delta\left(\frac{\log^{1/2} N}{N^{1/2}}\right).
\end{equation}
This result was obtained in \cite{gustavsson} from a corresponding central limit for the fixed-energy eigenvalue counting function $\# (\Sigma \cap (-\infty,x))$ for various energies $x$ in the bulk.

The heuristic \eqref{lambdai-improv} is further supported by the results of Dallaporta \cite[Theorem 5]{dallaporta}, who in our notation\footnote{See Section \ref{notation-sec} for our asymptotic notation conventions.} and normalization conventions established\footnote{Again, analogous results at the edge, or in intermediate regions between the bulk and edge, are also established in \cite{dallaporta}.} the second moment estimate
\begin{equation}\label{second}
  \E |\lambda_i -\sqrt{2N} \gamma_{i/N}|^2 \ll_\delta \frac{\log N}{N}
\end{equation}
for $i$ in the bulk region \eqref{bulk-def}, so in particular we have the mean estimate
$$ \E \lambda_i = \sqrt{2N} \gamma_{i/N} + O_\delta\left(\frac{\log^{1/2} N}{N^{1/2}}\right)$$
and the variance bound
$$ \Var \lambda_i \ll_\delta \frac{\log^{1/2} N}{N^{1/2}}.$$
Indeed, by comparing with the results of Gustavsson, we in fact have a matching lower bound for the variance as well:
$$ \Var \lambda_i \gg_\delta \frac{\log^{1/2} N}{N^{1/2}}.$$
A further result \cite[Proposition 6]{dallaporta} of Dallaporta then gives the concentration inequality
\begin{equation}\label{lconc}
  \P\left( |\lambda_i - \sqrt{2N} \gamma_{i/N}| \geq u/\sqrt{N} \right) \ll \exp\left( - \frac{c_\delta u^2}{u+\log N}\right)
\end{equation}
for some $c_\delta>0$ and all $u>0$.  Among other things, this allows one to also obtain higher moment estimates
\begin{equation}\label{moment}
  \E |\lambda_i -\sqrt{2N} \gamma_{i/N}|^p \ll_{\delta,p} \frac{\log^{p/2} N}{N^{p/2}}
\end{equation}
for any $p > 0$; see the end of \cite[\S 1.1]{dallaporta}; as another consequence of this concentration inequality we see that for every $A>0$ one has the eigenvalue rigidity estimate
\begin{equation}\label{rigidity}
  \lambda_i = \sqrt{2N} \gamma_{i/N} + O_{A,\delta}\left(\frac{\log N}{\sqrt{N}} \right)
\end{equation}
outside of an event of probability $O(N^{-A})$. Again, these fixed-index results were obtained as a consequence of fixed-energy analogues, which can be in turn established from the theory of determinantal processes.

\subsection{Eigenvalue gaps in the bulk}

Now we turn to eigenvalue gaps $\lambda_{i+1}-\lambda_i$, again staying in the bulk region\footnote{As is well known, eigenvalue gaps at the edge are instead described by the Tracy--Widom law \cite{tracy}.} for simplicity.  From the smoothness of $\rhosc$ in the bulk region, it is easy to see from \eqref{gammai-def} that
$$ \gamma_{i+m} = \gamma_{i/N} + \frac{m}{N \rhosc(\gamma_{i/N})} + O_\delta\left( \frac{\log^{O(1)} N}{N^2}\right)$$
whenever $i$ is in the bulk region \eqref{bulk-def} and $m = O(\log^{O(1)} N)$.  Comparing this asymptotic with \eqref{lambdai-improv} or \eqref{rigidity} leads the informal heuristic
\begin{equation}\label{ni}
  g_i \approx 1,
\end{equation}
where the \emph{normalized gap} $g_i$ is defined as
\begin{equation}\label{gap-def}
g_i \coloneqq \sqrt{N/2} \rhosc(\gamma_{i/N}) (\lambda_{i+1}-\lambda_i).
\end{equation}
However, we caution that the derivation of \eqref{ni} is not rigorous because a naive application of the triangle inequality here leads to an error term of $O(\sqrt{\log N})$ which overwhelms the right-hand side of \eqref{ni}; for instance, \eqref{lconc} and the triangle inequality only gives the bound
\begin{equation}\label{gap-weak}
  \P\left( g_i \geq u \right) \ll \exp\left( - \frac{c_\delta u^2}{u+\log N}\right)
\end{equation}
for $u>0$, which heuristically corresponds to an upper bound $g_i \lessapprox \sqrt{\log N}$.

More generally, we are led to the heuristic
\begin{equation}\label{ni-m}
  g_i +g_{i+1}+\dots+g_{i+m-1} \approx m
\end{equation}
for $m = O(\log^{O(1)} N)$, though again due to the logarithmic loss in \eqref{lambdai-improv} or \eqref{rigidity}, this latter estimate is currently only rigorous for $m \gg \sqrt{\log N}$ (or $m \gg \log N$, if one wants a failure probability of $O(N^{-A})$).

Nevertheless, in \cite{Tao-gap} it was shown that the normalized gap $g_i$
converges in probability in the vague topology to the Gaudin distribution $p(y)\ dy$, which can be for instance defined by the formula
$$ p(y) = \frac{d^2}{dy^2} \det(1 - K_{\mathrm{Sine}})_{L^2([0,y])}$$
where the determinant here is a Fredholm determinant, and $K_{\mathrm{Sine}}: \R \times \R \to \R$ is the \emph{Dyson sine kernel}
$$ K_{\mathrm{Sine}}(x,y) \coloneqq \frac{\sin(\pi(x-y))}{\pi(x-y)}$$
interpreted here as an integral operator on $L^2([0,y])$.  In other words, we have
\begin{equation}\label{gi-u}
  \P( g_i \geq u ) = \int_u^\infty p(y)\ dy + o(1)
\end{equation}
for any fixed $u>0$.  As in previous results, a fixed-energy version of this result was already known from the theory of determinantal processes \cite{deift}; in order to pass from this to the fixed-index assertion in \cite{Tao-gap}, it was necessary to show the eigenvalue gap $\lambda_{i+1}-\lambda_i$ was \emph{universal} in the sense that it did not change under small perturbations of $i$.  This was achieved in that paper by a more advanced application of determinantal process techniques.  A much more general universality result was subsequently established in \cite[Theorem 2.2]{erdos-yau}, which in our context asserts that for any fixed $k$ and fixed compactly supported Lipschitz $F \colon \R^k \to \R$, and any $i,j$ in the bulk region \eqref{bulk-def}, that
\begin{equation}\label{universal}
  \E F( g_i, g_{i+1}, \dots, g_{i+k-1} )
= \E F( g_j, g_{j+1}, \dots, g_{j+k-1} ) + O_{\delta,F,k}(N^{-c})
\end{equation}
for some absolute constant $c>0$.  In contrast to the previous methods, the arguments here did not rely on determinantal process techniques, instead using the local convergence to equilibrium of the Dyson Browian motion, as well as H\"older regularity theory for (discrete) parabolic equations.  As a consequence, their argument applied to more general Wigner matrices than GUE (and the two expectations in this universality relation could be applied to two different Wigner models).  Notably, a power saving $N^{-c}$ in the error term is now obtained in \eqref{universal}, which was not available in previous results such as \eqref{gi-u} (largely due to the fact that the variance of eigenvalue counting functions was only logarithmic in size).  On the other hand, the result in \eqref{universal} does not directly apply to observables $F$ that are not of compact support, or which depend on a number of variables $k$ that grow with $N$; as we shall discuss later, such extensions are desirable when studying an extension of the GUE eigenvalue process, namely the GUE minor process.  While in principle it seems possible to obtain some extensions of the bound \eqref{universal} to such settings, the arguments in \cite{erdos-yau} are rather lengthy, and we did not attempt to modify these arguments to handle these broader cases.

We mention two further results about the distribution of gaps.  The gap theorem from \cite[Theorem 19]{taovu}, when specialized to GUE, asserts that for every $c_0>0$ there exists $c_1>0$ such that
\begin{equation}\label{gap-thm}
 \P( g_i < N^{-c_0} ) \ll_\delta N^{-c_1}
\end{equation}
for any $i$ in the bulk region \eqref{bulk-def}.  This result was in fact established for a more general class of Wigner matrices; for the specific case of GUE, this result can also be established from fixed-energy calculations and the eigenvalue rigidity
estimate \eqref{rigidity}, using factors of $N^{-c_0}$ to absorb the logarithmic loss in that rigidity estimate.  We leave the details to the interested reader.  Finally, we mention that for every $c_0>0$ there exists $c_1>0$ such that one has the covariance estimate
$$ \Cov( P_1(g_i), P_2(g_j) ) \ll_{\delta,c_0} \|P_1\|_{C^5} \|P_2\|_{C^5} N^{-c_1} $$
for all $i,j$ in the bulk \eqref{bulk-def} with $|i-j| \geq N^{c_0}$ and all smooth compactly supported $P_1,P_2 \colon \R \to \R$; see \cite[Proposition 3.3(1)]{cipolloni}.  By a truncation argument together with \eqref{gap-weak}, one can then establish\footnote{This covariance bound can also be established more directly from \cite[Lemmas 4.2, 4.3]{cipolloni} (L\'aszl\'o Erd\H{o}s, private communication).}
$$ \Cov( g_i, g_j ) \ll_{\delta,c_0} N^{-c_1/2} $$
for the same $i,j$; see \cite{narayanan} for an application of this estimate.  In \cite{cipolloni}, estimates of this type were also used to obtain various ``quenched universality'' results for GUE and Wigner matrices; we refer readers to that paper for further details.

\subsection{New results}

In this paper we establish the following additional bounds on the distribution of eigenvalue gaps in the bulk, that further support the heuristics \eqref{ni}, \eqref{ni-m}.

\begin{theorem}[Gap bounds]\label{gap-bounds} Let $i$ lie in the bulk region \eqref{bulk-def} for some fixed $\delta > 0$.
  \begin{itemize}
    \item[(i)] (Concentration inequality) For any $0 < h \ll \log\log N$, we have $\P( g_i \geq h) \ll_\delta \exp(-h/4)$.
    \item[(ii)] (Moment bound) For any $p>0$, we have $\E g_i^p  \ll_{p,\delta} 1$.
    \item[(iii)] (Lower tail bound)  For any $h > 0$, we have $\P( g_i \leq h) \ll_\delta h^{2/3} \log^{1/2} \frac{1}{h}$.
    \item[(iv)] (Local eigenvalue rigidity) For any natural number $0 < m \leq \log^{O(1)} N$ and $\alpha>0$, one has
    \begin{equation}\label{imi}
      \P ( |g_i + \dots + g_{i+m-1} - m| > \alpha) \ll_\delta \frac{\log^{4/3} (2+m)}{\alpha^2}
    \end{equation}
    and
    \begin{equation}\label{imi-2}
      \E ( |g_i + \dots + g_{i+m-1} - m|^2 ) \ll_\delta \log^{7/3} (2+m).
    \end{equation}

  \end{itemize}
\end{theorem}

We do not expect these bounds to be sharp; the ``Wigner surmise'' (see, e.g., \cite{tracy}) predicts that the $\exp(-h/4)$ decay in (i) should instead be Gaussian in nature, and Gaudin kernel asymptotics (or GUE eigenvalue repulsion) similarly predicts that the $h^{2/3} \log^{1/2} \frac{1}{h}$ bound in (ii) should instead be $h^3$.  The range of $h$ for which the estimate (i) is applicable should also be able to be enlarged.  We also expect the right-hand side of (iv) to be improved, in analogy with \eqref{second}.  Nevertheless, these bounds suffice for our main application below; the key feature is that the bounds do not contain factors of $\log N$ on the right-hand side (in contrast to what would happen if one were to apply, for instance, \eqref{rigidity}), and also exhibit only logarithmic growth in the $m$ parameter.

From the moment bound (ii) and a standard truncation argument, one can now remove the compact support hypothesis from \eqref{universal}, at the cost of making the error terms qualitative.  More precisely, if $k$ is fixed and $F: \R^k \to \R$ is a continuous function of polynomial growth, we now have
\begin{equation}\label{smash}
\E F( g_i, g_{i+1}, \dots, g_{i+k-1} )
= \E F( g_j, g_{j+1}, \dots, g_{j+k-1} ) + o(1)
\end{equation}
for $i,j$ in the bulk region \eqref{bulk-def} for some fixed $\delta>0$.  Indeed, one can truncate $F$ to a compact set and approximate by a Lipschitz continuous function, using (ii) to control the error and \eqref{universal} to control the main term, and then take limits in the usual fashion.  In particular, the mixed moments of $g_i,\dots,g_{i+k-1}$ only differ by $o(1)$ from those of $g_j,\dots,g_{j+k-1}$; for instance, one now has
$$ \Var g_i = \Var g_j + o(1)$$
and
$$ \Cov(g_i, g_{i+m}) = \Cov(g_j, g_{j+m}) + o(1)$$
for any fixed $m$.

We now remark on the methods of proof of Theorem \ref{gap-bounds}, which we prove in Section \ref{rotarch-sec}.  Following \cite{Tao-gap}, we shall rely on the theory of determinantal processes, combined with the classical Plancherel--Rotach asymptotics for the determinantal kernel of the GUE process.  This theory is good at counting the number $\#(\Sigma \cap J)$ of (suitably normalized) eigenvalues in a given interval $J$, where $\Sigma$ is a normalized version of the GUE eigenvalue process; indeed, as is well known \cite{hough}, this number is the sum of independent Bernoulli variables.  However, to obtain control on gaps of a fixed index, this is not sufficient; one needs to get some control on the \emph{joint} distribution of two eigenvalue counts $\#(\Sigma \cap J)$, $\#(\Sigma \cap I)$ for two nested intervals $I \subset J$.  Here, due to subtle correlations between these two random variables, one cannot interpret this joint distribution as a sum of independent variables any more.  Nevertheless, it is still possible to use the description of $\Sigma \cap J$ as a mixture of finite rank determinantal processes, combined with more refined estimates on sums of independent Bernoulli variables, to obtain satisfactory upper bounds on the distribution, which is enough to obtain the conclusions of Theorem \ref{gap-bounds}.

\subsection{Application to the GUE minor process}

Recall that $\lambda_1 < \dots < \lambda_N$ are the eigenvalues of an $N \times N$ GUE matrix $H$.  If we let $\lambda'_1 < \dots < \lambda'_{N-1}$ be the eigenvalues of the top left $N-1 \times N-1$ minor of this matrix (which is again a GUE matrix, but with $N$ replaced by $N-1$), then we almost surely have the Cauchy interlacing law
$$ \lambda_i < \lambda'_i < \lambda_{i+1}$$
for $1 \leq i < N$.  The combined process $\tilde \Sigma \subset \R \times \{N-1,N\}$ defined by
\begin{equation}\label{td-def}
  \tilde \Sigma \coloneqq \{ (\lambda_i, N): 1 \leq i \leq N\} \cup \{ (\lambda'_i, N-1): 1 \leq i \leq N-1\}
\end{equation}
forms two rows of the \emph{GUE minor process}.  Like the GUE eigenvalue process, this process is known to be determinantal \cite{adler}, \cite{johansson-nordenstam}; we describe the kernel for this process below in Section \ref{lowfreq-sec}.  Among other things, it was shown in \cite{adler} that under suitable rescaling, this kernel in the bulk converges to the kernel of a Boutillier bead process \cite{boutillier} (with a drift parameter determined by the location in the bulk).  Informally, this asserts that the (rescaled) GUE minor process converges to the Boutillier bead process in a \emph{fixed energy} sense.  

As an application of the above estimates, we may now also obtain a universal law in the \emph{fixed index} sense.  Define the \emph{interlacing gaps}
$$ \tilde g_i \coloneqq \sqrt{N/2} \rhosc(\gamma_{i/N}) (\lambda'_i-\lambda_i)$$
then we have $0 < \tilde g_i < g_i$.  Informally, the quantities $g_i,\dots,g_{i+m}, \tilde g_i, \dots, \tilde g_{i+m}$ then describe the behavior of the GUE minor process near the index $i$.

Our main result for two rows of the minor process is then

\begin{theorem}[Interlacing universality]\label{interlacing-univ} Let $i, j$ lie in the bulk region \eqref{bulk-def} for some fixed $\delta > 0$ with $i = j + o(N)$, let $m$ be fixed, and let $F: \R^{2m} \to \R$ be a fixed smooth, compactly supported function.  Then
  $$ \E F(g_i,\dots,g_{i+m}, \tilde g_i,\dots,\tilde g_{i+m}) = \E F(g_j,\dots,g_{j+m}, \tilde g_j,\dots,\tilde g_{j+m}) + o(1)$$
as $N \to \infty$.
\end{theorem}

In principle, we may now average in the $i$ parameter and derive fixed index limiting laws for such statistics as the interlacing gaps $\tilde g_i$ in terms of corresponding statistics for the bead process, although a precise formalization of this principle is tricky due to the lack of a natural definition of ``fixed index'' for the bead process (which has an infinite number of elements).

We prove Theorem \ref{interlacing-univ} in Section \ref{interlacing-sec}.  The strategy is to leverage the existing universality result \eqref{smash}, which for instance would establish Theorem \ref{interlacing-univ} if $F$ depended only on the minor gaps $\tilde g_i,\dots,\tilde g_{i+m}$ and not on the eigenvalue gaps $g_i,\dots,g_{i+m}$.  To recover the full strength of Theorem \ref{interlacing-univ}, we use Schur's complement to obtain a probabilistic equation relating the latter gaps to the former; however, initially this equation involves an unbounded number of minor gaps $\tilde g_j$ rather than a bounded number.  However, the various bounds in Theorem \ref{gap-bounds}, together with routine Taylor approximation, can be used to largely eliminate the influence of all but a bounded number of these minor gaps, effectively replacing allowing us to approximate the expectation $\E F(g_i,\dots,g_{i+m}, \tilde g_i,\dots,\tilde g_{i+m})$ to acceptable error with an expectation $\E F'(\tilde g_{i-m'},\dots,\tilde g_{i+m'})$ for some suitable $m' > m$ and some suitable test function $F'$, at which point one can apply \eqref{smash}.

The above result was for two rows of the minor process, but an adaptation of the argument also applies to any fixed number of rows of the minor process as follows.  For any $1 \leq M \leq N$, and let $\lambda^{(M)}_1 < \dots < \lambda^{(M)}_{M}$ be the eigenvalues of the top left $N-k \times N-k$ minors of $H$, thus in the previous notation we have $\lambda^{(N)}_i = \lambda_i$ and $\lambda^{(N-1)}_i = \lambda'_i$.  We then have the normalized eigenvalue gaps
$$ g^{(M)}_i \coloneqq \sqrt{M/2} \rhosc(\gamma_{i/M}) (\lambda^{(M)}_{i+1}-\lambda^{(M)}_i)$$
as well as the interlacing gaps
$$ \tilde g^{(M)}_i \coloneqq \sqrt{M/2} \rhosc(\gamma_{i/M}) (\lambda^{(M-1)}_{i}-\lambda^{(M)}_i)$$
for any $1 \leq i < M \leq N$.
Thus, for instance, $g^{(N)}_i = g_i$ and $\tilde g^{(N)}_i = g'_i$.  Also, for any fixed $m$ and any $1 \leq i \leq i+m+1 \leq N$, $g_i,\dots,g_{i+m},\tilde g_i,\dots,\tilde g_m$ can be expressed as a linear combination (with bounded coefficients) of the $g^{(M)}_j$ for $M=N-1,N$ and $i \leq j \leq i+m$, together with a single interlacing gap $\tilde g^{(N)}_i$. We can then generalize Theorem \ref{interlacing-univ} to

\begin{theorem}[Interlacing universality, II]\label{interlacing-univ-2} Let $i, j$ lie in the bulk region \eqref{bulk-def} for some fixed $\delta > 0$ with $i = j + o(N)$, let $k, m \geq 1$ be fixed, and let $F: \R^{(m+1)k + k-1} \to \R$ be a fixed smooth, compactly supported function.  Then
  $$ \E F( (g^{(M)}_{i'})_{N-k < M \leq N; i \leq i' \leq i+m}, (\tilde g^{(M)}_i)_{N-k+1 < M \leq N} ) = \E F( (g^{(M)}_{j'})_{N-k < M \leq N; j \leq j' \leq j+m}, (\tilde g^{(M)}_j)_{N-k+1 < M \leq N} ) + o(1)$$
as $N \to \infty$.
\end{theorem}
The bound \eqref{universal} (with the polynomial error term replaced by the qualitative error $o(1)$) is then equivalent to the $k=0$ case of Theorem \ref{interlacing-univ-2}, while Theorem \ref{interlacing-univ} is equivalent to the $k=1$ case, as can be seen by applying an appropriate linear transformation.

In principle, the above universality results allow one to control various fixed index statistics, by replacing them with fixed energy analogues which are more computable, though in practice the calculations can get rather involved.  As an application of this type, we obtain some non-trivial reduction in variance of linear statistics of interlacing gaps:

\begin{theorem}[Interlacing statistics variance]\label{thm-stats} Let $i$ lie in the bulk region \eqref{bulk-def} for some fixed $\delta > 0$, let $m \geq 1$ be fixed, and let $a_1,\dots,a_m$ be complex numbers.  Then
$$ \Var \sum_{l=1}^m a_l \tilde g_{i+l} \ll_{\delta,A} \left(\frac{m}{\log^A(2+m)}+o(1)\right) \sum_{l=1}^m |a_l|^2$$
for any $A>0$.
\end{theorem}

Note that a direct application of Theorem \ref{gap-bounds}(ii), combined with the triangle inequality and the Cauchy--Schwarz inequality, yields the weaker bound
$$ \Var \sum_{l=1}^m a_l \tilde g_{i+l} \ll m \sum_{l=1}^m |a_l|^2.$$
The key point is then the asymptotic improvement of the bound by an arbitrary power of $\log(2+m)$.

Theorem \ref{thm-stats} will be used in forthcoming work of Narayanan \cite{narayanan-limit} to obtain limit laws for random hives with GUE boundary 
data.

The proof of Theorem \ref{thm-stats} is rather involved, and occupies Section \ref{variance-sec}.  The first step is to use the triangle inequality to reduce to establishing a ``low frequency estimate''
\begin{equation}\label{low-eq} 
  \Var \sum_{l=1}^m \tilde g_{i+l} \ll m^{2-c} + o(1)
\end{equation}
and a ``high frequency estimate''
\begin{equation}\label{high-eq}
  \Var \sum_{l=1}^m a_l (\tilde g_{i+l} - \tilde g_{i+l-1}) \ll_{\delta,A} \left(\frac{m}{\log^A(2+m)}+o(1)\right) \sum_{l=1}^m |a_l|^2.
\end{equation}
We focus first on the high frequency estimate \eqref{high-eq}.  By a further application of the triangle inequality, we can pass from interlacing gaps to eigenvalue gaps, reducing to establishing a bound of the form
$$ \Var \sum_{l=1}^m a_l (g_{i+l} - g_{i+l-1}) \ll_{\delta,A} \left(\frac{m}{\log^A(2+m)}+o(1)\right) \sum_{l=1}^m |a_l|^2.$$
By some Fourier analytic manipulations, it then suffices to obtain a bound of the form
$$ \E |\sum_{l=1}^m e(\xi l) g_{i+l}|^2 \ll_{\delta,A} \frac{m^2}{\log^A(2+m)} + o(1)$$
for various frequencies $\xi$ that are not extremely close to zero.  Using the universality result in \eqref{smash}, one can pass from this fixed index estimate to a fixed energy estimate, eventually reducing to establish a bound of the form
$$ \int_0^m (m-h) \E e(\xi N_{x,0,h})\ dh \ll_{\delta,A} \frac{m^2}{\log^A(2+m)} + o(1)$$
where $N_{x,0,h}$ is the number of eigenvalues in a certain interval of length $\frac{h}{N \rhosc(x)}$ (so that the expectation of $N_{x,9,h}$ is roughly $h$).  When the frequency $\xi$ is large, one can bound each expectation $\E e(\xi N_{x,0,h})$ separately, using the fact that $N_{x,0,h}$ is distributed as the sum of independent Bernoulli variables.  However, this method does not work well when $\xi$ is small.  In this case, we instead perform a Taylor expansion of the exponential and end up having to control expressions of the form
$$ \int \psi(h/m) \E e(\xi h) \binom{N_{x,0,h}}{j}\ dh$$
for fixed $j$ and some smooth compactly supported $\psi$.  This expression can be approximated using Plancherel--Rotach asymptotics and Fourier expansion, and requires some delicate combinatorial analysis of the different linear combinations of frequencies that arise from that expansion, in order to exploit cancellation from the $e(\xi h)$ phase factor.

For the low frequency estimate \eqref{low-eq}.  Here, we can use the universality result from Theorem \ref{interlacing-univ} to replace this fixed index estimate with a fixed energy estimate.  It turns out that the main task is to obtain a non-trivial bound on 
$$ \Var \frac{1}{m} \int_0^{m} F(s)\ ds$$
where $F$ is a normalized indicator function of the interlacing intervals $\bigcup_j [\lambda_j,\lambda'_j]$.  This can be done by using the determinantal kernel for the minor process, together with Plancherel--Rotach asymptotics and classical manipulations of Hermite polynomials.

\subsection{Acknowledgements}

The author is supported by NSF grant DMS-2347850.  We thank Hariharan Narayanan for suggesting this problem, and Jun Yin for helpful discussions.

\section{Notation}\label{notation-sec}

We use $o(1)$ to denote any quantity that goes to zero as the size $N$ of the GUE matrix goes to infinity (while holding ``fixed'' parameters such as $m$ constant). We also use the usual asymptotic notation $X = O(Y)$, $X \ll Y$, or $Y \gg X$ to denote the bound $|X| \leq CY$ for some constant $C$ and sufficiently large $N$; if we need $C$ to depend on additional parameters, we indicate this by subscripts, thus for instance $X \ll_\delta Y$ denotes the bound $|X| \leq C_\delta Y$ for some constant $C_\delta$ depending on $\delta$ and all sufficiently large $N$. We also write $X \asymp Y$ for $X \ll Y \ll X$ (with the same subscripting conventions).

We write $e(\theta) \coloneqq e^{2\pi \sqrt{-1} \theta}$ for any real $\theta$.

\section{Sums of independent Boolean variables}

As mentioned in the introduction, the random variable $\#(\Sigma \cap I)$ has the sum of independent Bernoulli variables.  For future reference, we therefore record some standard calculations regarding such sums.

\begin{lemma}[Sums of independent Boolean random variables]\label{lem-xi}  Let $\xi_i$ be a finite collection of independent Boolean variables, with each $\xi_i$ having mean $\E \xi_i = \lambda_i \in [0,1]$.

\begin{itemize}
  \item[(i)] $\sum_i \xi_i$ has mean $\sum_i \lambda_i$ and variance $\sum_i \lambda_i (1-\lambda_i)$.  Furthermore, we have the Bernstein inequality
  \begin{equation}\label{bern}
    \P\left( \left|\sum_i \xi_i - \sum_i \lambda_i\right| \geq t\right) \leq 2\exp\left( - c\min\left( \frac{t^2}{\sum_i \lambda_i (1-\lambda_i)}, t\right)\right)
  \end{equation}
  for every $t>0$ and an absolute constant $c>0$.
  \item[(ii)] One has $\P(\sum_i \xi_i = 0) \leq \exp( - \sum_i \lambda_i )$.
  \item[(iii)] One has $\P(\sum_i \xi_i > 1) \leq \frac{1}{2} ((\sum_i \lambda_i)^2 - \sum_i \lambda_i^2)$.
  \item[(iv)]  One has
  \begin{equation}\label{llt-1}
    \sup_n \P\left( \sum_i \xi_i = n \right) \ll \left(\sum_i \lambda_i\right)^{-1/2}.
  \end{equation}
  \item[(v)] More generally, if $S$ is a subset of the indices $i$, and $F((\xi_i)_{i \not \in S})$ is a non-negative random variable that depends only on the $\xi_i$ with $i \not \in S$, then
  \begin{equation}\label{llt-2}
     \sup_n \E 1_{\sum_i \xi_i = n} F((\xi_i)_{i \not \in S}) \ll \left(\sum_{i \in S} \lambda_i\right)^{-1/2} \E F((\xi_i)_{i \not \in S}).
  \end{equation}
\end{itemize}
\end{lemma}

  \begin{proof}
    The claims in (i) follow from linearity of expectation and additivity of variance for indepenent random variables, as well as the classical Bernstein inequality (see, e.g., \cite[Theorem 2.8.1]{vershynin}).  For (ii), the left-hand side can be explicitly written as $\prod_i (1-\lambda_i)$, and the claim follows since $0 \leq 1-\lambda_i \leq \exp(-\lambda_i)$.  For (iii), we bound
    \begin{align*}
      \P(\sum_i \xi_i > 1) &\leq \frac{1}{2} \E \left(\left(\sum_i \xi_i\right)^2 - \left(\sum_i \xi_i\right)\right) \\
      &= \frac{1}{2} \sum_{i,j: i \neq j} \lambda_i \lambda_j
    \end{align*}
    from which the claim follows.

    Now we prove (iv). By Fourier expansion, one has
      $$  \P( \sum_i \xi_i = n ) = \int_{\R/\Z} \E e\left(\sum_i \xi_i \theta\right) e^{-\theta n}\ d\theta$$
      and hence by the triangle inequality and independence
      $$  \P( \sum_i \xi_i = n ) = \int_{\R/\Z} \prod_i |\E e(\xi_i \theta)|\ d\theta.$$
      Direct calculation shows that
      $$ \E e(\xi_i \theta) \leq \exp( - c \lambda_i \|\theta\|^2)$$
      for some absolute constant $c>0$, where $\|\theta\|$ denotes the distance of $\theta$ to the nearest integer.  Inserting this bound, we obtain (iv).

      Finally, we prove (v).  Observe from (iv) (and shifting $n$ as needed) that if one conditions on all the $\xi_i$ with $\xi_i \not \in S$, then the conditional probability that $\sum_i \xi_i$ is equal to $n$ is $O(  (\sum_{i \in S} \lambda_i)^{-1/2})$.  The claim now follows from the law of total probability.
  \end{proof}

  \begin{remark} The estimate in Lemma \ref{lem-xi}(ii) is useful when the mean $\sum_i \lambda_i$ is large, while the estimate in Lemma \ref{lem-xi}(iii) is similarly useful when the mean $\sum_i \lambda_i$ is small.   Intuitively, the estimates in (iv), (v) assert that the sum of independent Boolean variables do not concentrate at any scale smaller than the standard deviation $(\sum_{i \in S} \lambda_i)^{-1/2}$ arising from some subset $S$ of indices, even when considering arbitrary weights involving the other $\xi_i$ with $i \not \in S$.
  \end{remark}

  \section{Determinantal process calculations}

  In this section we record some general determinantal process calculations that will be applied in the next section to the GUE eigenvalue process.

  Let $\Omega = (\Omega,\mu)$ be a locally compact Polish measure space equipped with a Radon measure $\mu$ (in practice, $\Omega$ will be a subset of $\R$ and $\mu$ will just be Lebesgue measure).  We define a \emph{determinantal kernel} to be a symmetric measurable function $K: \Omega \times \Omega \to \R$ such that the integral operator
  $$ P_K f(x) \coloneqq \int_{\Omega} K(x,y) f(y)\ dy$$
  is well-defined as a non-negative contraction on $L^2(\Omega)$ (note that the hypothesis that $K$ is bounded ensures that $P_K$ is locally trace class).  If furthermore $P_K^2 = P_K$, we say that the determinantal kernel is a \emph{projection kernel}.  We say that the determinantal kernel is \emph{finite rank} if the operator $P_K$ is finite rank.  By the spectral theorem, this is equivalent to having the representation
  \begin{equation}\label{kxy}
    K(x,y) = \sum_i \lambda_i \phi_i(x) \phi_i(y)
  \end{equation}
  for some finite index set $i$, some eigenvalues $\lambda_i \in [0,1]$, and some orthonormal functions $\phi_i \in L^2(\Omega)$.

  It is known (see e.g., \cite[Theorem 22]{hough}) that every kernel $K$ defines a determinantal point process $\Sigma$ on $\Omega$.  In the case of a finite rank kernel, we can view this point process by introducing independent Bernoulli random variables $\xi_i \in \{0,1\}$ with $\E \xi_i = \lambda_i$, and considering the random projection kernel
  $$ K_\xi(x,y) \coloneqq \sum_i \xi_i \phi_i(x) \phi_i(y).$$
  Then $K_\xi$ also determines a determinantal point process $\Sigma_\xi$ on $\Omega$, with a fixed cardinality
  \begin{equation}\label{sigma-card}
    \# \Sigma_\xi = \sum_i \xi_i,
  \end{equation}
  and $\Sigma$ is a mixture of the $\Sigma_\xi$, thus
  \begin{equation}\label{mixture}
     \P( \Sigma \in E ) = \E \P(\Sigma_\xi \in E | \xi)
  \end{equation}
  for any event $E$ on the space of point processes; see \cite[Theorem 7]{hough}.  In particular, we have
  $$ \# \Sigma \equiv \sum_i \xi_i.$$
  Applying Lemma \ref{lem-xi}, we conclude the expectation formula
  \begin{equation}\label{exp}
    \E \# \Sigma = \sum_i \lambda_i = \tr P_K,
  \end{equation}
  the variance formula
  \begin{equation}\label{var}
    \Var \# \Sigma = \sum_i \lambda_i (1-\lambda_i) = \tr P_K(1-P_K),
  \end{equation}
  the vanishing bound
  \begin{equation}\label{van-2}
    \P( \# \Sigma = 0 ) \leq \exp( - \sum_i \lambda_i ) = \exp(-\tr P_K),
  \end{equation}
  and the repulsion bound
  \begin{equation}\label{repulse}
    \P( \# \Sigma > 1 ) \leq \frac{1}{2} \left(( \sum_i \lambda_i )^2 - \sum_i \lambda_i^2\right) = \frac{1}{2} ((\tr(P_K))^2 - \tr P^2_K).
  \end{equation}
  In terms of the kernel $K$, we have
  \begin{equation}\label{exp-ker}
   \sum_i \lambda_i  = \tr(P_K) = \int_\Omega K(x,x)\ d\mu(x)
  \end{equation}
  and
  \begin{equation}\label{var-ker}
    \sum_i \lambda_i (1-\lambda_i) = \tr(P_K (1-P_K)) = \int_\Omega K(x,x)\ d\mu(x) - \int_\Omega \int_\Omega K(x,y)^2\ d\mu(x) d\mu(y).
  \end{equation}

  Now we study how such a process restricts to a measurable subset $I$ of $\Omega$.

  \begin{theorem}\label{det-est}  Let $K$ be a finite rank determinantal kernel on $\Omega$, and let $I$ be a measurable subset of $\Omega$. Let $1_I: L^2(\Omega) \to L^2(\Omega)$ be the restriction operator to $I$.  Define the quantities
    \begin{align}
      A &\coloneqq \Var \# \Sigma = \tr( P_K (1-P_K) ) \label{A-1} \\
      B &\coloneqq \Var \# (\Sigma \cap I) = \tr( 1_I P_K 1_I (1-P_K) ) \label{B-1}
    \end{align}
    and assume the lower bound
  \begin{equation}\label{A-2}
    \tr( P_K^2 (1-P_K^2) ) \gg A.
  \end{equation}
  We also assume $A$ to be non-zero. Let $n$ be a natural number.

  \begin{itemize}
    \item[(i)] We have
    $$ \P( \# \Sigma=n ) \ll A^{-1/2}.$$
    \item[(ii)] We have
    $$ \P( \# \Sigma=n \wedge \# (\Sigma \cap I) = 0) \ll A^{-1/2} \exp( - \tr( 1_I P_K) + O(B) ).$$
    \item[(iii)] We have
    $$ \P( \# \Sigma=n \wedge \# (\Sigma \cap I) > 1) \ll A^{-1/2} \tr( 1_I P_K)^2.$$
    \item[(iv)] We have
    $$ \E 1_{\# \Sigma=n} (\#(\Sigma \cap I) - \tr(1_I P_K))^2 \ll A^{-1/2} B (1+B).$$
  \end{itemize}
  \end{theorem}

  The key point here is that any restriction of the form $\# \Sigma = n$ gains a factor in the estimates proportional to the standard deviation $A^{1/2}$ in the cardinality of the point process $\Sigma$.  This will allow us to use the theory of determinantal processes to control eigenvalues at a fixed index with satisfactory estimates, since fixing the index is related to fixing the number of elements of a point process on a half-line $(-\infty,E)$.  The lower bound \eqref{A-2} is a technical condition, which roughly speaking asserts that the non-projection behavior of $P_K$ is dominated  by eigenvalues $\lambda_i$ that are far from both $0$ and $1$, as opposed to eigenvalues very close to $0$ or very close to $1$; for the processes arising from GUE matrices, this condition will be verified in practice.  

  \begin{proof}
  From the self-adjoint nature of $P_K$, the self-adjoint projection nature of $1_I$, and the cyclic property of trace we have the identity
  $$
  \tr( 1_I P_K 1_I (1-P_K) )  = \tr( P_K 1_I (1-P_K) )
  + \tr( 1_I P_K (1-1_I) (1_I P_K (1-1_I))^* )$$
  and hence by \eqref{B-1} we have
  \begin{equation}\label{B-2}
     \tr( 1_I P_K (1-P_K) ) \leq \tr( 1_I P_K 1_I (1-P_K) ) = B.
  \end{equation}

  It will be convenient to introduce the local eigenvector inner products
  $$ c_{ij} \coloneqq \int_I \phi_i \phi_j\ d\mu.$$
  We expand $K$ using the spectral theorem as \eqref{kxy}, then the conditions \eqref{A-1}, \eqref{A-2}, \eqref{B-1}, \eqref{B-2} become
  \begin{align}
    \sum_i \lambda_i (1-\lambda_i) &= A \label{A-1-spec} \\
    \sum_i \lambda_i^2 (1-\lambda_i^2) &\gg A \label{A-2-spec} \\
    \sum_i \lambda_i c_{ii}\ d\mu - \sum_i \sum_j \lambda_i \lambda_j c_{ij}^2 &= B \label{B-1-spec} \\
    \sum_i \lambda_i (1-\lambda_i) c_{ii} &\ll B. \label{B-2-spec}
  \end{align}
  From \eqref{A-1-spec}, \eqref{A-2-spec}, we see that for a suitably small constant $1 \ll \eps \leq 1$, one has
  $$ \sum_i \lambda_i^2 (1-\lambda_i)^2 - \eps \lambda_i(1-\lambda_i) \gg A.$$
  The summand is $O(1)$, and hence must be positive for $\gg A$ values of $i$.  For such values, one has $\lambda_i (1-\lambda_i) \asymp_\eps 1$.  If we now allow implied constants to depend on $\eps$, we may thus find a collection $S$ of indices of cardinality
  \begin{equation}\label{S-card}
    \# S \asymp A
  \end{equation}
  such that
  \begin{equation}\label{lambda-s}
    \lambda_i (1-\lambda_i) \asymp 1
  \end{equation}
  for all $i \in S$.  The claim (i) now follows from Lemma \ref{lem-xi}(iv).

  Now we are ready to prove (ii).  By \eqref{mixture}, \eqref{sigma-card} the left-hand side can be written as
  $$ \E 1_{\sum_i \xi_i = n} \P\left( \# (\Sigma_\xi \cap I) = 0 | \xi\right)$$
  which by \eqref{van-2} and the spectral theorem is bounded by
  $$ \E 1_{\sum_i \xi_i = n} \exp\left( - \sum_i \xi_i c_{ii} \right).$$
  Meanwhile,
  \begin{equation}\label{tripk}
    \tr(1_I P_K) = \sum_i \lambda_i c_{ii},
  \end{equation}
  so it suffices to show that
  $$ \E 1_{\sum_i \xi_i = n} \exp\left( - \sum_i (\xi_i-\lambda_i) c_{ii} \right)
  \ll A^{-1/2} \exp(O(B)).$$
  From \eqref{lambda-s}, \eqref{B-2-spec} we have
  \begin{equation}\label{bas}
    \sum_{i \in S} (\xi_i-\lambda_i) c_{ii} \ll B,
  \end{equation}
  so it suffices to show that
  $$ \E 1_{\sum_i \xi_i = n} \exp\left( - \sum_{i \not \in S} (\xi_i-\lambda_i) c_{ii} \right)
  \ll A^{-1/2} \exp(O(B)).$$
  By Lemma \ref{lem-xi}(v), \eqref{lambda-s}, \eqref{S-card} it suffices to show that
  $$ \E  \exp\left( - \sum_{i \not \in S} (\xi_i-\lambda_i) c_{ii} \right)
  \ll \exp(O(B)).$$
  By independence, the left-hand side factors as
  $$ \prod_{i \not \in S} \E  \exp\left( - (\xi_i-\lambda_i) c_{ii} \right).$$
  Direct calculation shows that
  $$  \exp\left( - (\xi_i-\lambda_i) c_{ii} \right) = 1 + O\left( \lambda_i (1-\lambda_i) c_{ii} \right)$$
  and the claim (ii) now follows from \eqref{B-2-spec}.

  Now we prove (iii). Again we apply \eqref{mixture}, \eqref{sigma-card} to write the left-hand side as
  $$ \E 1_{\sum_i \xi_i = n} \P( \#(\Sigma_\xi \cap I) > 1 | \xi ) $$
  which by \eqref{repulse} and the spectral theorem is bounded by
  $$ \frac{1}{2} \E 1_{\sum_i \xi_i = n} \left(\left(\sum_i \xi_i c_{ii}\right)^2 - \sum_{i,j} \xi_i \xi_j c_{ij}^2\right)$$
  From \eqref{lambda-s} and the spectral theorem we have
  $$ \sum_{i \in S} c_{ii} \ll \tr 1_I P_K$$
  so we may bound
  \begin{align*}
  \left(\sum_i \xi_i c_{ii}\right)^2 &- \sum_{i,j} \xi_i \xi_j c_{ij}^2
  \leq O( (\tr 1_I P_K)^2 ) + O\left( (\tr 1_I P_K) \sum_{i \not \in S} \xi_i c_{ii} \right)\\
  &\quad + \left(\sum_{i \not \in S} \xi_i c_{ii}\right)^2 - \sum_{i,j \not \in S} \xi_i \xi_j c_{ij}^2 \\
  &\leq O( (\tr 1_I P_K)^2 ) + O\left( (\tr 1_I P_K) \sum_{i \not \in S} \sum_{i \not \in S} \xi_i c_{ii} \right)\\
  &\quad + \sum_{i,j \not \in S; i \neq j} \xi_i \xi_j c_{ii} c_{jj}
  \end{align*}
  where we have discarded some negative terms.  Applying Lemma \ref{lem-xi}(v) and the triangle inequality, it thus suffices to show that
  $$ \sum_{i \not \in S} \E \xi_i c_{ii} \ll \tr 1_I P_K$$
  and
  $$  \sum_{i,j \not \in S; i \neq j} \E \xi_i \xi_j c_{ii} c_{jj} \ll (\tr 1_I P_K)^2.$$
  But both of these follow by evaluating the expectation and then using \eqref{tripk}.

  Now we prove (iv).  Again we apply \eqref{mixture}, \eqref{sigma-card} to write the left-hand side as
  $$ \E 1_{\sum_i \xi_i = n} \E( (\#(\Sigma_\xi \cap I) - \tr(1_I P_K))^2 | \xi ) $$
  From the Pythagorean theorem, we have
  $$
  \E( (\#(\Sigma_\xi \cap I) - \tr(1_I P_K))^2 | \xi )
  = (\E(\#(\Sigma_\xi \cap I)|\xi) - \tr(1_I P_K))^2 + \Var(\#(\Sigma_\xi \cap I) | \xi)$$
  so it suffices to show that
  \begin{equation}\label{claim-1-e}
     \E 1_{\sum_i \xi_i = n} (\E(\#(\Sigma_\xi \cap I)|\xi) - \tr(1_I P_K))^2 \ll A^{-1/2} B(1+B)
  \end{equation}
  and
  \begin{equation}\label{claim-2-e}
    \E 1_{\sum_i \xi_i = n} \Var(\#(\Sigma_\xi \cap I) | \xi) \ll A^{-1/2} B(1+B).
  \end{equation}
  We begin with \eqref{claim-1-e}.  From \eqref{exp-ker} and the spectral theorem we have
  $$ \E(\#(\Sigma_\xi \cap I)|\xi) = \sum_i \xi_i c_{ii}$$
  and also
  $$ \tr(1_I P) = \sum_i \lambda_i c_{ii}.$$
  Thus by the triangle inequality it suffices to show that
  \begin{equation}\label{useful}
     \E 1_{\sum_i \xi_i = n} \left(\sum_i |\xi_i - \lambda_i| c_{ii}\right)^2 \ll A^{-1/2} B(1+B).
  \end{equation}
  From \eqref{bas}, Lemma \ref{lem-xi} one already has
  $$
     \E 1_{\sum_i \xi_i = n}\left(\sum_{i \in S} |\xi_i - \lambda_i| c_{ii}\right)^2 \ll A^{-1/2} B^2
  $$
  so it suffices to show that
  $$ \E 1_{\sum_i \xi_i = n} \left(\sum_{i \not \in S} |\xi_i - \lambda_i| c_{ii}\right)^2 \ll A^{-1/2} B(1+B)$$
  Using Lemma \ref{lem-xi}, \eqref{lambda-s}, \eqref{S-card} as before, it suffices to show that
  $$ \E \left(\sum_{i \not \in S} |\xi_i - \lambda_i| c_{ii}\right)^2 \ll  B(1+B).$$
  The summands are independent random variables of mean and variance $O( \lambda_i (1-\lambda_i) c_{ii} )$, so the claim follows from \eqref{B-2-spec}.

  Now we prove \eqref{claim-2-e}.  By \eqref{var-ker} and the spectral theorem we have
  $$ \Var(\#(\Sigma_\xi \cap I) | \xi)
  = \sum_i \xi_i c_{ii} - \sum_i \sum_j \xi_i \xi_j c_{ij}^2.$$
  If we write
  $$ V \coloneqq \sum_i \lambda_i c_{ii} - \sum_i \sum_j \lambda_i \lambda_j c_{ij}^2$$
  then
  we have
  $$ \Var(\#(\Sigma_\xi \cap I) | \xi)
  = V + \sum_i (\xi_i-\lambda_i) c_{ii} - \sum_i \sum_j (\xi_i \xi_j - \lambda_i \lambda_j) c_{ij}^2.$$
  We expand
  $$ \xi_i \xi_j - \lambda_i \lambda_j = (\xi_i-\lambda_i) \lambda_j + \lambda_i (\xi_j-\lambda_j) + (\xi_i - \lambda_i) (\xi_j - \lambda_j)$$
  and note that
  $$ \sum_i \sum_j (\xi_i - \lambda_i) (\xi_j - \lambda_j) c_{ij}^2
  = \int_I \int_I \left|\sum_i (\xi_i - \lambda_i) \phi_i(x) \phi_i(y)\right|^2\ d\mu(x) d\mu(y)$$
  is non-negative.  Thus we have
  $$ \Var(\#(\Sigma_\xi \cap I) | \xi)
  = V + \sum_i (\xi_i-\lambda_i) c_{ii} - 2 \sum_i \sum_j (\xi_i - \lambda_i) \lambda_j c_{ij}^2.$$
  By Bessel's inequality, we have
  $$ \sum_j c_{ij}^2 \leq c_{ii};$$
  bounding $\lambda_j$ by $1$, we conclude from the triangle inequality that
  $$ \Var(\#(\Sigma_\xi \cap I) | \xi)
  = V + 3 \sum_i |\xi_i-\lambda_i| c_{ii}.$$
  From \eqref{B-1-spec} and Lemma \ref{lem-xi} we have
  $$\E 1_{\sum_i \xi_i = n} V \ll A^{-1/2} B$$
  so the claim follows from \eqref{useful} and the triangle inequality. This concludes the proof of Theorem \ref{det-est}.
  \end{proof}

  Using Theorem \ref{det-est}, we can now obtain control on the joint distribution of point process counts for projection determinantal kernels.

  \begin{corollary}\label{det-cor}  Let $K$ be a finite rank projection determinantal kernel on $\Omega$, and let $I \subset J \subset \Omega$ be measurable.  Define the quantities
  \begin{align}
  A &\coloneqq \Var \#(\Sigma \cap J) = \int_J \int_{\Omega \backslash J} K(x,y)^2\ d\mu(y) d\mu(x)\label{A-1-alt} \\
  B &\coloneqq \Var \#(\Sigma \cap I) = \int_I \int_{\Omega \backslash I} K(x,y)^2\ d\mu(y) d\mu(x)\label{B-1-alt}
  \end{align}
  and assume the lower bound
  \begin{equation}\label{A-2-alt}
    \int_J \int_J \int_{\Omega \backslash J} \int_{\Omega \backslash J} K(x,y) K(x',y) K(x,y') K(x',y')\ d\mu(y') d\mu(y) d\mu(x') d\mu(x) \gg A.
  \end{equation}
We assume $A$ to be non-zero. Let $n$ be a natural number.
  \begin{itemize}
    \item[(i)] We have
    $$ \P( \#(\Sigma \cap J)=n) \ll A^{-1/2}.$$
    \item[(ii)] We have
    $$ \P( \#(\Sigma \cap J)=n \wedge \# (\Sigma \cap I) = 0) \ll A^{-1/2} \exp( - \tr( 1_I P_K) + O(B) ).$$
    \item[(iii)] We have
    $$ \P( \#(\Sigma \cap J)=n \wedge \# (\Sigma \cap I) > 1) \ll A^{-1/2} \tr( 1_I P_K)^2.$$
    \item[(iv)] We have
    $$ \E 1_{\#(\Sigma \cap J)=n} (\#(\Sigma \cap I) - \tr(1_I P_K))^2 \ll A^{-1/2} B(1+B).$$
  \end{itemize}
  \end{corollary}

  Intuitively, the interpretation of these estimates is that the imposition of the constraint $\#(\Sigma \cap J)=n$ allows the estimates on $\# (\Sigma \cap I)$ to improve by the natural factor $A^{1/2}$, which is the standard deviation of $\#(\Sigma \cap J)$.

  \begin{proof}  We apply Theorem \ref{det-est} with $\Omega$ replaced with $J$, and $K$ replaced with the restriction $K_J$ to $J \times J$.  Because $P_K^2=P_K$, we have
  $$ K(x,y) = \int_\Omega K(x,z) K(z,y)\ d\omega(z)$$
  for almost every $x,y$, and hence the integral kernel of $P_{K_J} - P_{K_J}^2$ is
  $$ K(x,y) - \int_J  K(x,z) K(z,y)\ d\omega(z) = \int_{\Omega \backslash J}  K(x,z) K(z,y)\ d\omega(z).$$
  Taking traces and Frobenius norms of $P_{K_J} - P_{K_J}^2$, we conclude that \eqref{A-1}, \eqref{A-2} follow from \eqref{A-1-alt}, \eqref{A-2-alt} respectively.  The hypothesis \eqref{B-1} similarly follows from \eqref{B-1-alt}, and the corollary is now immediate from Theorem \ref{det-est}.
  \end{proof}

  \section{Plancherel--Rotach asymptotics}\label{rotarch-sec}

We now apply the previous determinantal process theory to control GUE eigenvalue gaps.
 Let  $x$ be in the bulk region \eqref{bulk-def} for fixed $\delta>0$.  It will be convenient to introduce the affine rescaling
 $$ f_x(t) \coloneqq \sqrt{2N}\left(x + \frac{t}{N \rhosc(x)}\right)$$
for $t \in \R$, and the associated counting functions
$$ N_{x,t} \coloneqq \#(\Sigma \cap [f_x(t), +\infty))$$
and
$$ N_{x,t,h} \coloneqq N_{x,t} - N_{x,t+h} = \#(\Sigma \cap [f_x(t), f_x(t+h))).$$
These functions are related to the eigenvalues $\lambda_i$, since we have
\begin{equation}\label{lit}
   \lambda_i \geq f_x(t) \iff N_{x,t} \geq N-i
\end{equation}
for any $1 \leq i \leq n$ and $t \in \R$.

We also introduce the normalized GUE kernel
$$ K^{(N)}(x,y) \coloneqq \sqrt{2N} K_N\left(\sqrt{2N} x, \sqrt{2N} y\right).$$

The normalized density of states $N \rho_N(x) \coloneqq K^{(N)}(x,x)$ has the asymptotic
  $$N \rho_N(x) = N \rhosc(x) + \frac{1}{4\pi} \left(\frac{1}{x-1}-\frac{1}{x+1}\right) \cos \left( n \int_x^1 \rho_{\mathrm{sc}}(y)\ dy \right) + O_\delta(1/N);$$
  see \cite[(4.2)]{erc}.  Among other things, this can be used to yield the asymptotic
  \begin{equation}\label{expected}
    \E N_{x,0} = \int_x^\infty N \rho_N(y)\ dy = N \int_x^1 \rhosc(y)\ dy
  + O_\delta(\log N/N);
  \end{equation}
  see \cite[Lemma 2.1]{gustavsson}. In a similar fashion, we have
  $$ \E N_{x,0,h} = \int_x^{\frac{h}{N \rho_{\mathrm{sc}}(x)}} N \rho_N(y)\ dy = h + O_\delta(\log^{O(1)} N / N)$$
  whenever $0 < h \ll \log^{O(1)} N$.

  This asymptotic also yields the crude upper bound
  $$ K^{(N)}(x,x)\ll_\delta N$$
  in this bulk region $[-1+\delta,1-\delta]$, which by the positive definiteness of $K^{(N)}$ implies that
  $$ K^{(N)}(x,y) \ll_\delta N$$
  for $x,y$ in this region $[-1+\delta,1-\delta]$.  By using the Christoffel--Darboux formula as in \cite[p. 165]{gustavsson} we also obtain the bound
  $$ K^{(N)}(x,y) \ll_\delta \frac{1}{|x-y|}$$
  in this region, while if $y$ lies outside of the bulk region $[-1+\delta,1-\delta]$ the same analysis, together with the Hermite polynomial asymptotics in \cite[\S 4]{gustavsson} and standard Airy function asymptotics, yields the bounds
  $$ K^{(N)}(x,y) \ll_\delta \frac{1}{|1-y^2|^{1/4} y^{100}}\frac{1}{|x-y|}$$
  (say); indeed one can even get a gaussian type decay as $|y| \to \infty$, but we will not need such good decay here.  Putting all these estimates together, we obtain the combined bound
  \begin{equation}\label{xxy-2}
    K^{(N)}(x,y) \ll_\delta \frac{1}{|1-y^2|^{1/4} (1+|y|)^{100}} \min\left( N, \frac{1}{|x-y|} \right)
  \end{equation}
  whenever $x \in [-1+\delta,1-\delta]$ and $y \in \R$.

  In \cite[Lemma 2.3]{gustavsson} the variance bound
  $$ \Var N_{x,0} = \frac{1+o(1)}{2\pi^2}  \log(N(1-x^{3/2}))$$
  was demonstrated for $x$ in the bulk region \eqref{bulk-def}; in particular, for $N$ large we have
  \begin{equation}\label{variance}
    \Var N_{x,0} \asymp_\delta \log N
  \end{equation}
  or equivalently
  $$ \int_x^\infty \int_{-\infty}^x K^{(N)}(a,b)^2\ da db \asymp_\delta \log N.$$
  We also claim the related lower bound
  \begin{align*}
    &\int_x^\infty \int_x^\infty \int_{-\infty}^x \int_{-\infty}^x K^{(N)}(a,b) K^{(N)}(a,b') K^{(N)}(a',b) K^{(N)}(a',b')\ da da' db db'\\
    &\quad  \gg_\delta \log N.
  \end{align*}
  Indeed, by Fubini's theorem we can write the left-hand side as
  $$  \int_{-\infty}^x  \int_{-\infty}^x K^*(a,a')^2\ da da'$$
  where $K^*$ is the variant kernel
  $$ K^*(a,a') \coloneqq \int_x^\infty K^{(N)}(a,b) K^{(N)}(a',b)$$
  If we restrict $(a,a')$ to the region $R/N \leq |a-x|, |a-x| \leq 2R/N$ for some $1 \leq R \leq N^{1/2}$, one can use the asymptotics in \cite[\S 5]{gustavsson} to find that
  \begin{equation}\label{pra}
    K^{(N)}(a,b) = \frac{\sin(N \pi \rho_{\mathrm{sc}}(x) (a-b)) + O(1/\log N)}{\pi (a-b)}
  \end{equation}
  when $b \in [x,+\infty)$ with $|x-b| \ll R$, and from this, \eqref{xxy-2}, and some routine integration by parts one can obtain a bound of the form
  $$ |K^*(a,a')| \gg \frac{N}{R} ( |\cos( N \pi \rho_{\mathrm{sc}}(x) (a-a')) |- O(1/\log N))$$
  from which the claim follows by integrating over $a,a'$ in the indicated region and then summing dyadically over $R$.

  From \eqref{xxy-2}, we also obtain the bound
  $$ \int_I \int_{I^c} K^{(N)}(a,b)^2\ da db \ll_\delta \log(2+h)$$
  whenever $I$ is an interval of the form $[f_x(0)/\sqrt{2N}, f_x(h)/\sqrt{2N}]$ for some $x$ in the bulk region \eqref{bulk-def} and $0 < h \ll \log^{O(1)} N$.  Equivalently, we have
  \begin{equation}\label{varfx}
    \Var N_{x,0,h} \ll_\delta \log(2+h).
  \end{equation}

  Inserting these bounds into Corollary \ref{det-cor}, we conclude

  \begin{corollary}\label{bulk-cor} Let $x$ lie in the bulk region \eqref{bulk-def} for some fixed $\delta>0$, and let $0 < h \ll \log^{O(1)} N$.  Let $n$ be a natural number.
    \begin{itemize}
      \item[(i)] We have
      $$ \P( N_{x,0} =n) \ll_\delta \frac{1}{\sqrt{\log N}}.$$
      \item[(ii)] We have
      $$ \P( N_{x,0} = n \wedge N_{x,0,h} = 0) \ll_\delta \frac{1}{\sqrt{\log N}} \exp( - h / 2 ).$$
      \item[(iii)] We have
      $$ \P( N_{x,0} = n \wedge N_{x,0,h} > 1) \ll_\delta \frac{1}{\sqrt{\log N}} h^2.$$
      \item[(iv)] We have
      $$ \E 1_{N_{x,0}=n} (N_{x,0,h} - h)^2 \ll_\delta \frac{1}{\sqrt{\log N}} \log^2 (2+h).$$
    \end{itemize}
    \end{corollary}

Note here the gain of $\sqrt{\log N}$ in the estimates, which represents the standard deviation of the integrated density of states $N_{x,0}$.

We are now ready to prove Theorem \ref{gap-bounds}.    We allow all bounds to depend on $\delta$.

 If $1 \leq H \leq \log^{10} N$, then from \eqref{lit}, \eqref{expected}, \eqref{variance} and Chebyshev's inequality we have
$$ \P\left( \lambda_i \geq f_x(t + H \log^{1/2} N) \right) = \P \left( N_{x,H \log^{1/2} N} \geq N-i \right) \ll 1 / H^2$$
and
$$ \P\left( \lambda_i < f_x(t - H \log^{1/2} N) \right) = \P \left(N_{x,-H \log^{1/2} N} < N-i \right) \ll 1 / H^2.
$$
If instead $H > \log^{10} N$, the same claims follow (with much stronger bounds) from \eqref{lconc}.  Finally the claim is trivial for $0 < H \leq 1$.  Thus one has the concentration bounds
\begin{equation}\label{lifx}
   \P\left( |\lambda_i - f_x(0)| \geq \frac{H \log^{1/2} N}{N \rhosc(x)} \right) \ll 1 / H^2
\end{equation}
for all $H > 0$.

To prove (i), we may assume without loss of generality that $h \geq 1$.  Applying \eqref{lifx} with $H \coloneqq \exp(h/4)$, we see that with an acceptable loss of probability we may restrict to the event that
\begin{equation}\label{hln}
  |\lambda_i - f_x(0)| < \frac{H \log^{1/2} N}{N \rhosc(x)}.
\end{equation}
If $g_i \geq h$, then from \eqref{lifx} we have the event
$$ \left(N_{x,t} = N-i\right) \wedge \left(N_{x,t, h/2} = 0\right)$$
holding for an interval of $t \in [-2H\log^{1/2} N,2H\log^{1/2} N]$ of length at least $h/2$, thus
$$ \int_{-2H\log^{1/2} N}^{2H\log^{1/2} N} 1_{N_{x,t} = N-i} 1_{N_{x,t,h/2} = 0}\ dt \geq h/2.$$
From Corollary \ref{bulk-cor}(ii) and Fubini's theorem, the left-hand side has expectation $O( H \exp(-h/2 )) = O(\exp(-h/4))$, and the claim follows from Markov's inequality.

Now we prove (ii).  Fix $p$, and let $A>0$ be sufficiently large depending on $p$.  From part (i), we have
$$ \P( g_i \geq 2^j) \ll_{A,p} \exp(-2^j/4)$$
whenever $2^j \leq A \log\log N$.  In particular, we have
$$ \P( g_i \geq 2^j) \ll_{A,p} \log^{-A/4} N$$
for $A \log\log N < 2^j \leq A \log N$.  Finally, from \eqref{gap-weak} we have
$$ \P( g_i \geq 2^j) \ll_{A,p} \exp(-2^{j/2})$$
for $2^j > A \log N$.  Multiplying these bounds by $2^{pj}$ and summing, one obtains the claim (ii).

For (iii), we may assume without loss of generality that $h \leq 1$.  
We apply \eqref{lifx} with $H \coloneqq h^{-1/3}$, so with an acceptable loss of probability we may restrict to the event that \eqref{hln} holds.  If $N \rho_{\mathrm{sc}}(x) (\lambda_{i+1} - \lambda_i) \leq h$, then from \eqref{lifx} we have the event
$$ \left(N_{x,t} = N-i\right) \wedge \left(N_{x,t,2h} > 1\right)$$
holding for an interval of $t \in [-2H\log^{1/2} N,2H\log^{1/2} N]$ of length at least $h$, thus
$$ \int_{-2H\log^{1/2} N}^{2H\log^{1/2} N} 1_{N_{x,t} = N-i} 1_{N_{x,t,2h} > 1}\ dt \geq h.$$
If $h \gg N^{-1}$, then $f_x(t)$ stays inside the bulk region, and  from Corollary \ref{bulk-cor}(iii) and Fubini's theorem, the left-hand side has expectation $O( H h^2 ) = O(h^{5/3})$.  In the regime $h \ll N^{-1}$, we instead discard the $1_{N_{x,t} = N-i}$ factor and use \eqref{repulse}, \eqref{xxy-2} to obtain a bound $\P( \# N_{x,t,2h} > 1 ) \ll h^2$, which gives the expectation bound $O((H \log^{1/2} N) h^2) = O(h^{5/3} \log^{1/2} \frac{1}{h})$. The claim (iii) follows from Markov's inequality.

Now we prove the first estimate \eqref{imi} of (iv).  We first dispose of an easy case in which $\alpha > m^2$.  In this case we use (ii) and the triangle inequality to conclude that
$$ \E ( g_i + \dots + g_{i+m-1})^p \ll_{p`} m^p $$
for any $p>1$, hence by Markov's inequality
$$ \P( |g_i + \dots + g_{i+m-1}-m| > \alpha) \ll_{p} m^p / \alpha^p.$$
Setting $p=4$ (say), we obtain the desired estimate for $\alpha > m^2$.  Hence we may assume that $\alpha \leq m^2$; in particular, we have $\alpha \ll \log^{O(1)} N$.

From \eqref{gap-def} one has
$$ g_{i+j} = (1+O(\log^{O(1)} N/N)) \sqrt{N/2} \rhosc(\gamma_{i/N}) (\lambda_{i+j+1}-\lambda_{i+j})$$
for $0 \leq j < m$, hence on telescoping we have
$$ g_1+\dots+g_{i+m-1} = (1+O(\log^{O(1)} N/N)) \sqrt{N/2} \rhosc(\gamma_{i/N}) (\lambda_{i+m}-\lambda_{i}).$$
It thus suffices to show that
$$\P \left( |\sqrt{N/2} \rhosc(\gamma_{i/N}) \left(\lambda_{i+m}-\lambda_{i}\right) - m| > \alpha\right) \ll \frac{\log^{4/3} (2+m)}{\alpha^2}.$$
Applying \eqref{lifx} with $H \coloneqq \alpha / \log^{2/3}(2+m)$, we may assume without loss of generality that
$$ |\lambda_i - f_x(0)| < \frac{H \log^{1/2} N}{N \rhosc(x)}.$$
We begin with the upper estimate
$$\P \left( \sqrt{N/2} \rhosc(\gamma_{i/N}) (\lambda_{i+m}-\lambda_{i}) - m > \alpha\right) \ll \frac{\log^{4/3} (2+m)}{\alpha^2}.$$
If the event
$$ \sqrt{N/2} \rhosc(\gamma_{i/N}) (\lambda_{i+m}-\lambda_{i}) - m > \alpha$$ holds, then one has
$$ \left( N_{t,x} = N-i\right) \wedge \left(N_{t,x,m+\alpha/2} < m\right)$$
for an interval of $t \in [-2H\log^{1/2} N,2H\log^{1/2} N]$ of length at least $\gg \alpha$, thus
$$ \int_{-2H\log^{1/2} N}^{2H\log^{1/2} N} 1_{N_{t,x} = N-i} 1_{N_{t,x,m+\alpha/2} <m}\ dt \gg \alpha.$$
By Corollary \ref{bulk-cor}(iv) and Markov's inequality, the integrand has expectation $O_\delta( \frac{1}{\sqrt{\log N}} \log^2 (2+m) / \alpha^2 )$.  By Fubini's theorem and Markov's inequality, the above event then occurs with probability $O( H \log^2 (2+m) / \alpha^3 )$, which is acceptable by choice of $H$.

A similar argument works for the lower estimate
$$\P ( \sqrt{N/2} \rhosc(\gamma_{i/N}) (\lambda_{i+m}-\lambda_{i}) - m < -\alpha) \ll \frac{\log^{4/3} (2+m)}{\alpha^2}.$$
We may assume that $\alpha \leq m$, as the left-hand side vanishes otherwise.  If this event holds, then one has
$$ \left(N_{t,x} = N-i\right) \wedge \left(N_{t,x,m-\alpha/2} \geq m\right)$$
for an interval of $t \in [-2H\log^{1/2} N,2H\log^{1/2} N]$ of length at least $\gg \alpha$, thus
$$ \int_{-2H\log^{1/2} N}^{2H\log^{1/2} N} 1_{N_{t,x} = N-i} 1_{N_{t,x,m-\alpha/2} \geq m}\ dt \gg \alpha.$$
One now repeats the previous arguments to obtain the required bound \eqref{imi}.

Now we prove \eqref{imi-2}.  By dyadic decomposition we may bound
$$ \E ( g_i + \dots + g_{i+m-1} - m|^2 ) \ll 1 + \sum_{\alpha \geq 1} \alpha^2 \P( |g_i + \dots + g_{i+m-1} - m| > \alpha ) $$
where the sum is over powers of two.  The contribution of all $\alpha \leq m^2$ is then acceptable from \eqref{imi}.  For the remaining values of $\alpha$, we use Markov's inequality and (ii) to bound
$$\P( |g_i + \dots + g_{i+m-1} - m| > \alpha ) \ll \alpha^4 \E |g_i + \dots + g_{i+m-1}|^4
\ll \alpha^4 m^4;$$
inserting this bound, we see that the contribution of the $\alpha > m^2$ is also acceptable.

\section{Analysis of interlacing gaps}\label{interlacing-sec}

We now prove Theorem \ref{interlacing-univ}.  We allow implied constants to depend on $\delta$.

The main idea is to relate the eigenvalues $\lambda_i$ of a GUE matrix with the eigenvalues $\lambda'_j$ of an $N-1 \times N-1$ minor, together with some additional gaussian random variables via a well known Schur complement procedure.  The bounds in Theorem \ref{gap-bounds} will allow one to approximately localize the resulting equation to effectively control each eigenvalue $\lambda_i$ by a bounded number of minor  eigenvalues $\lambda'_j$.  This will allow us to derive a universality result for the eigenvalue and interlacing gaps from a universality result of the eigenvalue gaps alone.

We turn to the details, if we write
$$ M = \begin{pmatrix} M' & X \\ X^* & a_{NN} \end{pmatrix}$$
where $M'$ is the top left $N-1 \times N-1$ minor of $M$, $X$ is a random complex Gaussian vector (with entries $N(0,1/4)_\C$), and $a_{NN}$ is a random Gaussian with distribution $N(0,1/2)_\R$, then for each eigenvalue $\lambda_i$ of $M$, an application of Schur's complement to the singular matrix $M - \lambda_i I$ reveals (almost surely) that
$$ a_{NN} - \lambda_i -X^* (M'-\lambda_i I)^{-1} X = 0.$$
Diagonalizing $M'$, we obtain the equation (cf., \cite[Lemma 40]{taovu})
\begin{equation}\label{ani}
   a_{NN} -\lambda_i - \sum_{j=1}^{N-1} \frac{|X_j|^2}{\lambda'_j - \lambda_i} = 0
\end{equation}
where $X_j$ are the components of $X$ with respect to an orthonormal eigenbasis of $M'$.  Crucially, the $X_j$ are independent copies of $N(0,1/4)_\C$ that are independent of both $a_{NN}$ and $M'$ (and hence of the $\lambda'_j$).

Almost surely, the $X_j$ are non-zero, and then the left-hand side of \eqref{ani} is a rational function of $\lambda_i$ with poles at the $\lambda'_j$ that are decreasing on the real line away from such poles.  In particular, $\lambda_i$ is the unique solution to \eqref{ani} in the interval $(\lambda'_{i-1}, \lambda'_i)$ if we hold the $X_j$ and $a_{NN}$ fixed.

Let $m_0$ be a fixed natural number; we will later send it slowly to infinity as $N \to \infty$.  Call an event \emph{small probability} if it occurs with probability $O(m_0^{-0.1}) + o(1)$.  Let $i$ be in the bulk region \eqref{bulk-def}.  We will approximate the equation \eqref{ani} in this case by a more localized equation involving a much smaller number of the $\lambda'_j$, , outside of an event of small probability.

By \eqref{rigidity} and the union bound, we have outside of an event of small probability that
$$ \lambda_j = \sqrt{2N} \gamma_{j/N} + O\left(\frac{\log N}{\sqrt{N}} \right)$$
for all $j$ in the bulk region \eqref{bulk-def}. Indeed, from \cite[Theorem 7, Proposition 10]{dallaporta} and the union bound we may in fact obtain the more general bound
$$ \lambda_j = \sqrt{2N} \gamma_{j/N} + O\left(\frac{\log N}{N^{1/6} \min(j, N-j+1)^{1/3}} \right)$$
for all $1 \leq j \leq N$.  Similarly we have
$$ \lambda'_j = \sqrt{2(N-1)} \gamma_{j/(N-1)} + O\left(\frac{\log N}{N^{1/6} \min(j, N-j)^{1/3}} \right)$$
for all $1 \leq j \leq N-1$.  By interlacing we also have
\begin{equation}\label{lam}
   \lambda_i = \sqrt{2(N-1)} \gamma_{i/(N-1)} + O\left(\frac{\log N}{\sqrt{N}} \right)
\end{equation}
In the regime $|j-i| \geq \log^{2} N$, this implies from Taylor expansion that
$$ \frac{1}{\lambda'_j- \lambda_i} = \frac{1}{\sqrt{2(N-1)} ( \gamma_{j/(N-1)} - \gamma_{i/(N-1)})}
+ O\left(\frac{N^{5/6} \log N}{|i-j|^2 \min(j, N-j)^{1/3}} \right).$$
By Markov's inequality (and noting that each $|X_j|^2$ has mean $1/2$), the quantity\footnote{For the rest of this section, all sums will be over the $j$ variable.}
$$ \sum_{|j-i| \geq \log^{2} N} |X_j|^2 \frac{N^{5/6} \log N}{|i-j|^2 \min(j, N-j)^{1/3}}$$
can be seen to have mean $O(\log^{-1} N)$, hence by Markov's inequality is equal to $o(1)$ outside of an event of small probability.  This lets us partially localize \eqref{ani} to
$$
a_{NN} -\lambda_i - \sum_{|j-i| < \log^{2} N} \frac{|X_j|^2}{\lambda_i - \lambda'_j}
- \sum_{|j-i| \geq \log^{2} N} \frac{|X_j|^2}{\sqrt{2(N-1)} (\gamma_{j/(N-1)} - \gamma_{i/(N-1)} )}
= o(1).$$

The latter sum can be calculated to have a variance of $O(N \log^{-2} N)$ and mean
\begin{equation}\label{mui}
  \mu_i \coloneqq \sum_{|j-i| \geq \log^{2} N} \frac{1}{2\sqrt{2(N-1)} (\gamma_{j/(N-1)} - \gamma_{i/(N-1)})}.
\end{equation}
By Chebyshev's inequality, this sum is then equal to $\mu_i+o(N^{1/2})$ outside of an event of small probability, thus
$$
a_{NN} -\lambda_i - \sum_{|j-i| < \log^{2} N} \frac{|X_j|^2}{\lambda'_j - \lambda_i}
- \mu'_i
= o(N^{1/2}).$$
The Gaussian random variable $a_{NN}$ is also of size $o(N^{1/2})$ outside of such an event, so by \eqref{lam} one has
$$ - \sqrt{2(N-1)} \gamma_{i/(N-1)} - \sum_{|j-i| < \log^{2} N} \frac{|X_j|^2}{\lambda'_j - \lambda_i}
- \mu_i
= o(N^{1/2})$$
outside of an event of small probability.

We have now localized the equation to only involve $O(\log^{2} N)$ values of $\lambda'_j$.  We now take advantage of Theorem \ref{gap-bounds} to localize further.  From Theorem \ref{gap-bounds}(iv)  and the union bound, we see that
$$ g_i + \dots + g_{i+k-1} = k + O( k^{0.6} )$$
and similarly
$$ g_{i-k} + \dots + g_{i-1} = k + O( k^{0.6} )$$
for all $m_0 \leq k \leq \log^2 N$ outside of an event of small probability.  By interlacing, this implies that
\begin{equation}\label{lli}
   \lambda'_j - \lambda_i, \lambda'_j - \lambda'_i, \lambda'_j - \lambda'_{i-1} =
\frac{j-i}{\sqrt{N/2} \rhosc(\gamma_{i/N})} + O\left( \frac{|i-j|^{0.6}}{\sqrt{N}}\right)
\end{equation}
whenever $m_0 \leq |j-i| \leq \log^2 N$.  Also, from Theorem \ref{gap-bounds}(ii), we have
$$ \lambda'_i-\lambda_{i-1} \ll m_0^{0.1} / \sqrt{N}$$
outside of an event of small probability.  We can thus use Taylor expansion to write
$$ \sum_{m_0 \leq |j-i| < \log^{2} N} \frac{|X_j|^2}{\lambda'_j - \lambda_i}
= \sum_{m_0 \leq |j-i| < \log^{2} N} \frac{|X_j|^2}{\lambda'_j- \lambda'_i }
+ O\left( m_0^{0.1} \sqrt{N} \sum_{m_0 \leq |j-i| < \log^{2} N} \frac{|X_j|^2}{|i-j|^2} \right).$$
The expression inside the $O()$ has mean $O(m_0^{-0.9} \sqrt{N})$, hence is $O(m_0^{-0.8} \sqrt{N})$ outside of an event of small probability.  Meanwhile, for fixed $\lambda'_j$, the random variable
$ \sum_{m_0 \leq |j-i| < \log^{2} N} \frac{|X_j|^2}{\lambda'_j - \lambda'_i}$ has variance $O(  m_0^{-1} \sqrt{N})$, hence is $O(m_0^{-0.4} \sqrt{N})$ within of its mean outside of an event of small probability.  We conclude that
$$ \sum_{m_0 \leq |j-i| < \log^{2} N} \frac{|X_j|^2}{\lambda'_j - \lambda_i }
= \sum_{m_0 \leq |j-i| < \log^{2} N} \frac{1}{2(\lambda'_j - \lambda'_i)} + O( m_0^{-0.4} \sqrt{N} )$$
outside of such a small event.  Applying \eqref{lli} and Taylor expansion, we then have
$$ \sum_{m_0 \leq |j-i| < \log^{2} N} \frac{|X_j|^2}{\lambda'_j - \lambda_i }
= \sum_{m_0 \leq |j-i| < \log^{2} N} \frac{\sqrt{N/2} \rhosc(\gamma_{i/N})}{2(i-j)} + O( m_0^{-0.4} \sqrt{N} ).$$
By symmetry, the sum on the right-hand side vanishes.  We conclude that, outside of an event of small probability,
$$ - \sqrt{2(N-1)} \gamma_{i/(N-1)} - \sum_{|j-i| < m_0} \frac{|X_j|^2}{\lambda'_j - \lambda_i}
- \mu_i
= O(m_0^{-0.4} N^{1/2}) + o(N^{1/2}).$$
From \eqref{gammai-def} one can calculate that
$$ \frac{1}{2\sqrt{2(N-1)} (\gamma_{j/(N-1)} - \gamma_{i/(N-1)})}
= (1+O(N^{-1/3})) \frac{\sqrt{N-1}}{2\sqrt{2}} \int_{\gamma_{(i-1)/(N-1)}}^{\gamma_{i/(N-1)}} \frac{\rhosc(x)}{x - \gamma_{i/(N-1)}}\ dx$$
(in fact sharper error estimates are available in the bulk) and hence by \eqref{mui}
$$ \mu_i = \frac{\sqrt{N-1}}{2\sqrt{2}} \int_{|x-\gamma_{i/(N-1)}| \geq \log^2 N / \rhosc(\gamma_{i/(N-1)}) \sqrt{2N}} \frac{\rhosc(x)}{x - \gamma_{i/(N-1)}}\ dx + o(\sqrt{N})$$
which we can simplify slightly to
$$ \mu_i = \frac{\sqrt{N-1}}{2\sqrt{2}} \mathrm{p.v.} \int_\R \frac{\rhosc(x)}{x - \gamma_{i/(N-1)} }\ dx + o(\sqrt{N}).$$
A standard computation using the Plemelj formula yields
$$ \int_\R \frac{\rhosc(x)}{t - x}\ dx = -2t$$
for $-1 \leq t \leq 1$, and so
$$ - \frac{\sqrt{2(N-1)}}{2} \gamma_{i/(N-1)} - \sum_{|j-i| < m_0} \frac{|X_j|^2}{\lambda'_j - \lambda_i}
= O(m_0^{-0.4} N^{1/2}) + o(N^{1/2}).$$
In terms of the gaps $g_i$, this can be rewritten as
$$ - \frac{\gamma_{i/(N-1)}}{\sqrt{2} \rhosc(\gamma_{i/(N-1)})} - \sum_{|j-i| < m_0} \frac{|X_j|^2}{g'_i + \dots + g'_{j-1} + \tilde g_i}
= O(m_0^{-0.4}) + o(1)$$
where we adopt the convention that $g'_i + \dots + g'_{j-1} \coloneqq -(g'_j + \dots g'_{i-1})$ if $j < i$.

The derivative of the left-hand side in $\tilde g_i$ is
$$ \sum_{|j-i| < m_0} \frac{|X_j|^2}{(g'_i + \dots + g'_{j-1} + \tilde g_i)^2}
\geq \frac{|X_i|^2}{\tilde g_i}.$$
Outside of a set of small probability, one has the bounds
\begin{equation}\label{xgg}
|X_i|^2 \geq m_0^{-0.1}; \quad m_0^{-0.2} \leq g'_{i-1} \leq m_0^{0.1}
\end{equation}
thanks to Theorem \ref{gap-bounds}(ii), (iii), and similarly that
\begin{equation}\label{xjj}
   |X_j|, g'_j \leq m_0
\end{equation}
(say) whenever  $|j-i| < m_0$, thanks to Theorem \ref{gap-bounds}(ii) and the union bound.  Hence the aforementioned derivative is at least $m_0^{-0.2}$.  By the mean value theorem, we conclude that if $g^*_i$ is the unique solution to the equation
$$ - \frac{\gamma_{i/(N-1)}}{\sqrt{2} \rhosc(\gamma_{i/(N-1)})} - \sum_{|j-i| < m_0} \frac{|X_j|^2}{g'_i + \dots + g'_{j-1} + g^*_i} = 0$$
then outside of a set of small probability we have
$$ \tilde g_i = g^*_i + O(m_0^{-0.2}) + o(1).$$
One can view $g^*_i$ as a function
$$ g^*_i = G\left(\frac{i}{N-1}, g'_{i-m_0},\dots,g'_{i+m_0}, X_{i-m_0},\dots,X_{i+m_0} \right).$$
From the inverse function theorem, we see that $G$ is a smooth function of the indicated variables in the indicated region \eqref{xgg}, \eqref{xjj}, that is bounded by $m_0^{0.1}$ thanks to the upper bound on $g'_{i-1}$.  Thus, by abuse of notation (and by applying a suitable smooth cutoff function), we can identify $G$ with a smooth, compactly supported function of size $O(m_0^{0.1})$.  If we now shift $i$ by to $m$, we see that outside of an event of small probability (where we now allow implied constants to depend on $m$), we have
\begin{equation}\label{tgij}
   \tilde g_{i+j} = G\left(\frac{i+j}{N-1}, g'_{i+j-m_0},\dots,g'_{i+j+m_0}, X_{i+j-m_0},\dots,X_{i+j+m_0} \right) + O(m_0^{-0.2}) + o(1)
\end{equation}
for all $j=0,\dots,m$.  Applying $F$ (and allowing implied constants to depend on $F$), we see that
$$ F(g_i,\dots,g_{i+m}, \tilde g_i,\dots,\tilde g_{i+m}) = \tilde G\left(\frac{i}{N-1}, g'_{i-m_0},\dots,g'_{i+m+m_0}, X_{i-m_0}, \dots, X_{i+m+m_0}\right) + O(m_0^{-0.1}) + o(1)$$
outside of an event of small probability, for some smooth compactly supported function $\tilde G$ (which depends on $F$ and $m, m_0$, but is independent of $N$ and $i$, and bounded by $O(1)$).  Taking expectations, we conclude that
$$ \E F(g_i,\dots,g_{i+m}, \tilde g_i,\dots,\tilde g_{i+m}) = \E \tilde G\left(\frac{i}{N-1}, g'_{i-m_0},\dots,g'_{i+m+m_0}, X_{i-m_0}, \dots, X_{i+m+m_0}\right) + O(m_0^{-0.1}) + o(1).$$
Conditioning out the gaussian variables $X_{i-m_0},\dots,X_{i+m+m_0}$, we conclude that
$$ \E F(g_i,\dots,g_{i+m}, \tilde g_i,\dots,\tilde g_{i+m}) = \E G_*\left(\frac{i}{N-1}, g'_{i-m_0},\dots,g'_{i+m+m_0}\right) + O(m_0^{-0.1}) + o(1)$$
for some smooth compactly supported function $G_*$ (that can depend on $m,m_0$, but not on $i$ or $N$).
Applying \eqref{universal}, we conclude that
$$ \E F(g_i,\dots,g_{i+m}, \tilde g_i,\dots,\tilde g_{i+m}) = \E F(g_j,\dots,g_{i+m}, \tilde g_i,\dots,\tilde g_{j+m})  + O(m_0^{-0.1}) + o(1)$$
for any fixed $m_0$.  Letting $m_0$ go sufficiently slowly to infinity, we obtain the claim.

\subsection{Extension to more rows}

We now indicate how the above arguments can be extended to handle more than one row of gaps.  Let $m \leq m_1 \leq \dots \leq m_{k-1}$ be fixed integers to be chosen later. From \eqref{tgij} and a linear transformation, we may write
$$ g^{(N)}_j = G_1( g^{(N-1)}_{j-m_1}, \dots, g^{(N-1)}_{j+m_1}, X_{j-m_1,1},\dots,X_{j+m_1,1}) + O(m_1^{-0.2}) + o(1)$$
for each $i \leq j \leq i+m$ outside of an event of probability $O_m(m_1^{-0.1}) + o(1)$, where the $X_{j,1}$ are independent copies of $N(0,1/4)_\C$ that are also independent of the $g^{(N-1)}_{j}$, and some smooth compactly supported function $G_1 \colon \R^{4m_1+2} \to \R$ that depends on $m_1$ but is independent of $N$ and $i$.  We also have
$$ \tilde g^{(N)}_i = \tilde G_1( g^{(N-1)}_{i-m_1}, \dots, g^{(N-1)}_{i+m_1}, X_{i-m_1,1},\dots,X_{i+m_1,1}) + O(m_1^{-0.2}) + o(1)$$
outside of a similar exceptional event, for a similar function $\tilde G_1$.   In a similar vein, we have
$$ g^{(N-1)}_j = G_2( g^{(N-2)}_{j-m_2}, \dots, g^{(N-2)}_{j+m_2}, X_{j-m_2,1},\dots,X_{j+m_2,1}) + O(m_2^{-0.2}) + o(1)$$
and
$$ \tilde g^{(N-1)}_i = \tilde G_2( g^{(N-2)}_{i-m_2}, \dots, g^{(N-2)}_{j+m_2}, X_{i-m_2,1},\dots,X_{i+m_2,1}) + O(m_2^{-0.2}) + o(1)$$
for each $i-m_1 \leq j \leq i+m+m_1$ outside of an event of probability $O_{m,m_1}(m_2^{-0.1}) + o(1)$, where $X_{j,2}$ are independent copies of $N(0,1/4)_\C$ that are also independent of the $g^{(N-2)}_{j}$, and some smooth compactly supported functions $G_2, \tilde G_2 \colon \R^{4m_2+2} \to \R$ that depend on $m_2$ but is independent of $N$ and $i$.  If we let each $m_i$ be sufficiently large depending on the previous $m_1,\dots,m_{i-1}$ and $m$, we may compse all these functional relationships, and obtain a unified system
$$ g^{(N-r)}_j = G_{r}( (g^{(N-k+1)}_{j+h})_{|h| \leq 2 m_{k-1}}, (X_{j+h,l})_{|h| \leq 2m_{k-1}, 1 \leq l \leq k-1} ) + O(m_1^{-0.1}) + o(1)$$
and
$$ \tilde g^{(N-r)}_i = G'_{r}( (g^{(N-k+1)}_{i+h})_{|h| \leq 2 m_{k-1}}, (X_{i+h,l})_{|h| \leq 2m_{k-1}, 1 \leq l \leq k-1} ) + O(m_1^{-0.1}) + o(1)$$
(say) of functional dependencies outside of an event of probability $O(m_1^{-0.05})+o(1)$, for all $0 \leq r < k-1$ and $i \leq j \leq i+m$, and some smooth compactly supported functions $G_r, G'_r$ that do not depend on $i$ or $N$.  We can now repeat the final part of the proof of Theorem \ref{interlacing-univ} to obtain Theorem \ref{interlacing-univ-2}.

\section{Variance of interlacing statistics}\label{variance-sec}

In this section we prove Theorem \ref{thm-stats}.  We allow all implied consants to depend on $\delta$. We first observe that it will suffice to establish the following two variants: a ``low-frequency'' estimate
\begin{equation}\label{low-freq}
   \Var \sum_{l=1}^m \tilde g_{i+l} \ll m^{2-c}+o(1)
\end{equation}
controlling averaged gaps, for some absolute constant $c>0$, and a ``high-frequency'' estimate
\begin{equation}\label{high-freq}
 \Var \sum_{l=1}^m a_l (\tilde g_{i+l} - \tilde g_{i+l-1}) \ll_A  \left(\frac{m}{\log^A(2+m)}+o(1)\right) \sum_{l=1}^m |a_l|^2
\end{equation}
controlling gap differences for all $A>0$.  Indeed, if we partition $\{1,\dots,m\}$ into $O(m/\log^{A/c}(2+m))$ intervals $I$ of length $\asymp \log^{A/c}(2+m)$, then we can apply \eqref{low-freq} to each interval and use the triangle inequality to obtain the bound
$$ \Var \sum_{l=1}^m a_l \tilde g_{i+l} \ll_A \left(\frac{m}{\log^A(2+m)}+o(1)\right)\sum_{l=1}^m |a_l|^2$$
whenever $a_l$ is constant on each of the intervals $I$.  If, on the other hand, $a_l$ has mean zero on each of the intervals $I$, then we can write $a_l = b_l - b_{l-1}$, where $b_0,\dots,b_m$ is a sequence of complex numbers with
$$ \sum_{l=0}^m |b_l|^2 \ll \left(\log^{2A/c}(2+m)+o(1)\right) \sum_{l=1}^m |a_l|^2$$
as can be seen by using the Cauchy--Schwarz inequality and the local mean zero hypothesis to obtain pointwise bounds on each $b_l$.  Applying \eqref{high-freq} with $a_l$ replaced by $b_l$ (and $A$ replaced by $2A/c+A$), and summing by parts (using Theorem \ref{gap-bounds}(ii) and interlacing to control boundary terms), we see that
$$ \Var \sum_{l=1}^m a_l \tilde g_{i+l} \ll_A \left(\frac{m}{\log^A(2+m)}+o(1)\right) \sum_{l=1}^m |a_l|^2$$
in this case.  Since a general sequence $a_l$ can be split into sequences of the two types considered above (and with $\ell^2$ norm not exceeding that of the original sequence), we obtain the claim by the triangle inequality.

\subsection{The high-frequency estimate}

We begin with the high-frequency estimate \eqref{high-freq}.  Here we can take advantage of the triangle inequality to reduce to controlling ordinary eigenvalue gaps, rather than interlacing gaps, as follows.  Up to negligible errors (that are easily treated by Theorem \ref{gap-bounds}(ii)), we can write $\tilde g_{i+l} - \tilde g_{i+l-1}$ as
$$ g'_{i+l-1} - g_{i+l-1},$$
where
$$ g'_i \coloneqq \sqrt{(N-1)/2} \rhosc(\gamma_{i/(N-1)}) (\lambda'_{i+1}-\lambda'_i)$$
is the $i^{\mathrm{th}}$ normalized gap for the $N-1 \times N-1$ top left minor of $M_N$.  Thus, by the triangle inequality, it suffices to establish the bounds
$$
  \Var \sum_{l=1}^m a_l g_{i+l-1} \ll_A \left(\frac{m}{\log^A(2+m)}+o(1)\right) \sum_{l=1}^m |a_l|^2
$$
and
$$
  \Var \sum_{l=1}^m a_l g'_{i+l-1} \ll_A \left(\frac{m}{\log^A(2+m)}+o(1)\right) \sum_{l=1}^m |a_l|^2.
$$
We just establish the former estimate here, as the latter is proven identically (just replace $N$ by $N-1$).  By repeating the previous arguments (and adjusting $c$ as necessary), we can reduce this former estimate in turn to the estimates
\begin{equation}\label{la}
\Var \sum_{l=1}^m g_{i+l-1} \ll (m^{2-c}+o(1))
\end{equation}
and
\begin{equation}\label{lb}
  \Var \sum_{l=1}^m a_l (g_{i+l}-g_{i+l-1}) \ll_A \left(\frac{m}{\log^A(2+m)}+o(1)\right) \sum_{l=1}^m |a_l|^2
\end{equation}
for some $c>0$.
The estimate \eqref{la} (with any $0 < c < 1$) follows from Theorem \ref{gap-bounds}(iv), so we turn to \eqref{lb}.  Expanding the square, the left hand side is
$$\sum_{1 \leq l,l' \leq m} a_l \overline{a_{l'}} \Cov(g_{i+l}-g_{i+l-1}, g_{i+l'}-g_{i+l'-1}).$$
By \eqref{smash}, we can write this expression as
$$\sum_{1 \leq l,l' \leq m} a_l \overline{a_{l'}} \Cov(g_{i+l-l'}-g_{i+l-l'-1}, g_{i}-g_{i-1})$$
up to negligible errors.  By Plancherel's theorem, it then suffices to show that
$$ \sum_{-m \leq h \leq m} e(\xi h) \Cov(g_{i+h}-g_{i+h-1}, g_{i}-g_{i-1}) \ll_A \frac{m}{\log^A(2+m)}+o(1)$$
for all $\xi \in\R$.  The cutoff $-m \leq h \leq m$ is not positive definite in $h$, but we can compensate for this by the following averaging argument. It will suffice to show that
$$ \sum_{M-m \leq h \leq M+m} \left(1-\frac{|h-M|}{m}\right) e(\xi h) \Cov(g_{i+h}-g_{i+h-1}, g_{i}-g_{i-1}) \ll_A \frac{m}{\log^A(2+m)}+o(1)$$
for any fixed $M,m \geq 1$, as the claim then follows by replacing $m$ by $m^{1/2}$ (say), averaging in $M$ from $-m$ to $m$, and using Theorem \ref{gap-bounds}(ii) to control the boundary terms.  Using \eqref{smash}, the left-hand side may be rewritten as
$$ m^{-1} e(\xi M) \Cov\left( \sum_{l=1}^m e(\xi l) (g_{i+M+l}-g_{i+M+l-1}), \sum_{l=1}^m e(\xi l) (g_{i+l}-g_{i+l-1})\right)$$
up to negligible errors, so by the Cauchy-Schwarz inequality it suffices to show that
$$ \Var\left(\sum_{l=1}^m e(\xi l) (g_{i+l}-g_{i+l-1})\right) \ll_A \frac{m^2}{\log^A(2+m)}.$$
Bounding the variance by the second moment, we reduce to showing that
$$ \E \left|\sum_{l=1}^m e(\xi l) (g_{i+l}-g_{i+l-1})\right|^2 \ll_A \frac{m^2}{\log^A(2+m)}.$$

By summation by parts (and using Theorem \ref{gap-bounds}(ii) to control boundary terms), it suffices to show that
$$ \E \left|\sum_{l=1}^m (e(\xi (l+1)) - e(\xi l)) g_{i+l}\right|^2 \ll_A \frac{m^2}{\log^A(2+m)}.$$
The left-hand side can be bounded by $m^2 \|\xi\|$ by Theorem \ref{gap-bounds}(ii) and the triangle inequality, where $\|\xi\|$ denotes the distance of $\xi$ to the nearest integer.  Thus we may assume we are in the not-too-low-frequency case
\begin{equation}\label{highfreq}
  \|\xi\| \geq \frac{1}{\log^A(2+m)}.
\end{equation}
By the triangle inequality, it now suffices to show that
\begin{equation}\label{geelong} \E \left|\sum_{l=1}^m e(\xi l) g_{i+l}\right|^2 \ll \frac{m^2}{\log^A(2+m)} + o(1).
\end{equation}
We now divide into three subcases: the relatively low frequency case
\begin{equation}\label{low}
  \frac{1}{\log^A(2+m)} \leq \|\xi\| \leq \frac{1}{\log^2(2+m)}
\end{equation}
the medium frequency case
\begin{equation}\label{medium}
  \frac{1}{\log^2(2+m)} \leq \|\xi\| \leq \frac{1}{\log^{1/4}(2+m)}
\end{equation}
and the very high frequency case
\begin{equation}\label{veryhigh}
\|\xi\| \geq \frac{1}{\log^{1/4}(2+m)}.
\end{equation}
Actually we can subsume the medium frequency case into the relatively low frequency case as follows.  If \eqref{medium} holds, then
we introduce an intermediate scale $m' \coloneqq \exp(\log^{0.1} m)$; the point being that the medium frequency $\xi$ behaves like a low frequency relative to $m'$ in the sense that
$$   \frac{1}{\log^{10A}(2+m')} \ll \|\xi\| \ll \frac{1}{\log^2(2+m')}.$$
Assuming that \eqref{geelong} has already been established for the relatively low frequency case, we have
$$
\E \left|\sum_{l=1}^{m'} e(\xi l) g_{i+l}\right|^2 \ll \frac{(m')^2}{\log^A(2+m')} + o(1).
$$
If one averages in $i$ over an interval of length $m$ and uses the triangle inequality, we obtain the desired bound \eqref{geelong} (using Theorem \ref{gap-bounds}(ii) to handle boundary terms, and increasing $A$ as necessary).  Thus we may assume that we are either in the case \eqref{low} or \eqref{veryhigh}.

For now we treat the two cases \eqref{low}, \eqref{veryhigh} in a unified fashion.  We now transition from a fixed index estimate to a fixed energy estimate as follows.
By \eqref{smash}, we can write \eqref{geelong} as
$$ \frac{1}{\log^{10} N} \sum_{i \leq j < i + \log^{10} N} \Var\left(\sum_{l=1}^m e(\xi l) g_{j+l}\right)$$
up to negligible error.  Thus it will suffice to show that
$$ \sum_{i \leq j < i + \log^{10} N} \Var\left(\sum_{l=1}^m e(\xi l) g_{j+l}\right) \ll_A \left(\frac{m^2}{\log^A(2+m)}+ o(1)\right) \log^{10} N.$$
Observe that $\sum_{l=1}^m e(\xi l) g_{j+l}$ is equal to
$$
\sqrt{N/2} \rhosc(\gamma_{i/N}) \sum_{l=1}^m e(\xi l) (\lambda_{j+l+1}-\lambda_{j+l})$$
up to negligible errors.  Thus it suffices to show that
\begin{equation}\label{complex} \sum_{i \leq j < i + \log^{10} N} \E\left|\sum_{l=1}^m e(\xi l) (\lambda_{j+l+1}-\lambda_{j+l})\right|^2 \ll_A \left(\frac{m^2}{\log^A(2+m)} + o(1)\right) \frac{\log^{10} N}{N}.
\end{equation}
We can write
$$\sum_{l=1}^m e(\xi l) (\lambda_{j+l+1}-\lambda_{j+l}) =
e(-\xi j) \int_{\lambda_j}^{\lambda_{j+m}} e(\xi N_x)\ dx$$
where $N_x \coloneqq \# ( \{\lambda_1,\dots,\lambda_N\} \cap (-\infty,x])$.  Thus we can write the left-hand side of \eqref{complex} as as
$$ \sum_{i \leq j < i + \log^{10} N} \E\left|\int_{\lambda_j}^{\lambda_{j+m}} e(\xi N_x)\ dx\right|^2$$
which one can rearrange as
$$ 2 \Re \E \int_\R \int_0^\infty e(\xi (N_{x+h}-N_x)) \# \{ i \leq j \leq i+\log^{10} N: \lambda_j \leq x \leq x+h \leq \lambda_{j+m} \}\ dh dx.$$
Using \eqref{rigidity}, we see that we can restrict $x$ to the region
$$ \sqrt{2N} \gamma_{(i+\log^5 N)/N} \leq x \leq \sqrt{2N} \gamma_{(i+\log^{10} N - \log^5 N)/N}$$
with acceptable error; one can similarly restrict $h$ to the region
$$ 0 \leq h \leq \frac{\sqrt{2N}}{N \rhosc(x)} \log^2 N.$$ 
Rescaling $h$ by $\sqrt{2N} / N \rhosc(x)$, we can thus write the previous expression as
\begin{align*}
  & \frac{2\sqrt{2N}}{N \rhosc(\gamma_{i/N})} \Re \E \int_{\sqrt{2N} \gamma_{(i+\log^5 N)/N}}^{\sqrt{2N} \gamma_{(i+\log^{10} N-\log^5 N)/N}} \int_0^{\log^2 N} e(\xi N_{x,0,h}) \\
  &\quad \# \{ i \leq j \leq i+\log^{10} N: \lambda_j \leq f_x(0) \leq f_x(h) \leq \lambda_{j+m} \}\ dh dx
\end{align*}
up to acceptable error, where $N_{x,0,h} \coloneqq \# ( \{\lambda_1,\dots,\lambda_N\} \cap (f_x(0),f_x(h)])$.  From \eqref{rigidity} again, we can see that outside of an event of probability $O(N^{-100})$, one has
$$ \# \{ i \leq j \leq i+\log^{10} N: \lambda_j \leq f_x(0) \leq f_x(h) \leq \lambda_{j+m} \} = (m-N_{x,0,h})_+$$
for all $x$ in the region of integration, so we can simplify the previous expression up to acceptable error as
$$
   \frac{2\sqrt{2N}}{N \rhosc(\gamma_{i/N})} \mathrm{Re} \E \int_{\sqrt{2N} \gamma_{i/N}}^{\sqrt{2N} \gamma_{(i+\log^{10} N)/N}} \int_0^{\log^2 N} e(\xi N_{x,0,h}) (m - N_{x,0,h})_+\ dh dx.$$
By the triangle inequality, it thus suffies to show the fixed energy estimate
$$ \E \int_0^{\log^2 N} e(\xi N_{x,0,h}) (m - N_{x,0,h})_+\ dh \ll_A \frac{m^2}{\log^A(2+m)} + o(1)$$
for all $x$ in the region of integration.

From \eqref{varfx} we see that for $h \geq 2m$, the random variable $N_{x,0,h}$ has mean $(1+o(1)) h$ and variance $O( \log(2 + h))$, so that $(m - N_{x,0,h})_+$ has expectation
$$ \ll \frac{m \log(2+h)}{h^2} + o(1)$$
by Chebyshev's inequality.  The contribution of this regime can easily be seen to be acceptable, so we may restrict to the regime $h \leq 2m$. In this regime, a similar argument using the Cauchy--Schwarz inequality gives
$$ \E|(m - N_{x,0,h})_+ - (m - h)_+|
\leq \E |N_{x,0,h} - h| \ll \log^{1/2}(2+m) + o(1)$$
and so we may replace $(m - N_{x,0,h})_+$ by $(m - h)_+$ up to acceptable errors.  By the triangle inequality, it thus suffices to establish the bound
\begin{equation}\label{mmh}
   \int_0^m (m-h) \E e(\xi N_{x,0,h})\ dh \ll_A \frac{m^2}{\log^A(2+m)} + o(1)
\end{equation}
for any $x$ in the bulk region $[-1+\delta/2,1-\delta/2]$.

It is now time to separate the two cases \eqref{low}, \eqref{veryhigh}.
We begin with the very high frequency case \eqref{veryhigh}.  Here it suffices to obtain a bound of the form
$$
\E e(\xi N_{x,0,h}) \ll_A \frac{1}{\log^A(2+m)} + o(1)
$$
for $\frac{m}{\log^A(2+m)} \leq h \leq m$.
By \eqref{sigma-card} and independence, the left-hand side can be written as
$$
\prod_r (1-\mu_r + \mu_r e(\xi))$$
where $\mu_r$ are the eigenvalues of $1_I P_K 1_I$, with $I \coloneqq (f_x(0), f_x(h)]$ and $K$ the Dyson sine kernel.  Routine calculations yield
$$ \log \frac{1}{|1-\mu_r + \mu_r e(\xi)|} \gg \mu_r (1-\mu_r) \|\xi\|^2$$
so by \eqref{veryhigh} it will suffice to show that
$$ \sum_r \mu_r (1-\mu_r) \gg \log(2+m) + o(1).$$
By \eqref{A-1-alt} the left-hand side is
$$ \int_I \int_{I^c} K(a,b)^2\ da db$$
which we can rescale as
$$ \int_0^h \int_{[0,h]^c} \left(\sqrt{N/2} \rhosc(x) K(\sqrt{N/2} \rhosc(x) a, \sqrt{N/2} \rhosc(x) b)\right)^2\ da db.$$
Using \eqref{xxy-2} we may neglect the contributions of those $a$ with $|a| \geq m^{100}$ (say).  Using Plancherel--Rotach asymptotics \eqref{pra}, this expression can then be calculated to be
$$ \int_0^h \int_{[0,h]^c} \frac{\sin^2(\pi(a-b))}{\pi^2 (a-b)^2}\ da db$$
up to negligible errors, and the claim then follows by a routine calculation.

\subsection{The very low-frequency estimate}

Now we consider the very low frequency case \eqref{low}.  As before we set $N_{x,0,h} \coloneqq \# (\Sigma \cap (f_x(0),f_x(h)])$. We write \eqref{mmh} as
$$ \int_0^m \left(1 - \frac{h}{m}\right) \E e(\xi N_{x,0,h})\ dh \ll_A \frac{m}{\log^A(2+m)} +o(1)$$
We may take $A$ to be an integer.  We take advantage of the concentration of the quantity $N_{x,0,h}$ around its mean $h$. Indeed,
from \eqref{bern} and the variance bound \eqref{varfx}, we have
$$ \E |N_{x,0,h}-h|^K \ll_K \log^K(2+m)$$
for any $K$.  To take advantage of this, we Taylor expand
$$ e(\xi N_{x,0,h}) = \sum_{j=0}^{A-1} \frac{e(\xi h)}{j!} (N_{x,0,h}-h)^j \xi^j + O_A( \|\xi\|^A (N_{x,0,h}-h)^A).$$
By the hypothesis \eqref{low}, the contribution of the error term will be acceptable, so by the triangle inequality it suffices to establish the bounds
\begin{equation}\label{psihm-basic}
 \int_0^m (1-h/m) \E e(\xi h) (N_{x,0,h} - h)^j\ dh \ll_A \frac{m}{\log^A(2+m)}+o(1)
\end{equation}
for any $0 \leq j \leq A-1$.  It will suffice to obtain the bounds
\begin{equation}\label{psihm}
   \int \psi(h/m) \E e(\xi h) j! \binom{N_{x,0,h}}{j}\ dh \ll_{A,j} m^{-A} \|\psi\|_{C^{K_{A,j}}}+o(1)
\end{equation}
for any $A>0$, any fixed $j$, any smooth function $\psi$ supported on $(0,1)$, and some $K_{A,j}$ depending on $A, j$, as the previous claim then follows by approximating the function $t \mapsto 1-t$ on $(0,1)$ by such a smooth function $\psi(t)$ (up to an $L^1$ error of size $O(\log^{-C_A}(2+m)))$ for some suitably large $C_A$, applying \eqref{psihm} (noting that the gain of $m^{-A}$ will overcome any logarithmic losses from the $C^{K_{A,j}}$ norm), and then controlling the error using the triangle inequality.

Fix $j$; we allow implied constants to depend on $j$.  We will not be able to establish this bound purely from the triangle inequality, but instead need to carefully extract some cancellation from the $e(\xi h)$ factor using the bounds in \eqref{low}.
By \eqref{det} and rescaling we may write
$$ j! \binom{N_{x,0,h}}{j} =  \int_0^h \dots \int_0^h \det\left(\sqrt{N/2} \rhosc(x) K(\sqrt{N/2} \rhosc(x) a_i, \sqrt{N/2} \rhosc(x) a_l)\right)_{1 \leq i,l \leq j}\ da_1 \dots da_j$$
and hence by Plancherel--Rotach asymptotics \eqref{pra} this is equal to
$$ \int_0^h \dots \int_0^h \det\left(\frac{\sin(\pi (a_i-a_l))}{\pi(a_i-a_l)}\right)_{1 \leq i,l \leq j}\ da_1 \dots da_j + o(1)$$
with the usual convention that $\frac{\sin \pi x}{\pi x}$ equals $1$ when $x=0$.  If we write
$$ \frac{\sin \pi x}{\pi x} = \E e(\theta x)$$
where $\theta$ is a uniform random variable in $[-1/2,1/2]$, then we can write the previous expression as
$$ h^j \E \det(e( \theta_i(a_i - a_l) h))_{1 \leq i,l \leq j}$$
where $a_1,\dots,a_j$ and $\theta_1,\dots,\theta_j$ are independent random variables, uniformly distributed on $[0,1]$ and $[-1/2,1/2]$ respectively. Performing a Leibniz expansion of the determinant, this is
$$ h^j \sum_{\sigma \in S_j} (-1)^{\sigma} \E e\left( \sum_{i=1}^j \theta_i(a_i-a_{\sigma(i)}) h\right)$$
where $(-1)^\sigma$ is $+1$ when $\sigma$ is an even permutation and $-1$ when $\sigma$ is an odd permutation.
By Fubini's theorem, the left-hand side of \eqref{psihm} can now be written as
$$ m^{j+1} \sum_{\sigma \in S_j} (-1)^{\sigma} \E \varphi\left( m \left(\sum_{i=1}^j \theta_i(a_i-a_{\sigma(i)}) + \xi\right) \right) + o(1)$$
where $\varphi$ is the Fourier transform of $t \mapsto t^j \psi(t)$:
$$ \varphi(y) \coloneqq \int t^j \psi(t) e(y t)\ dt.$$
From repeated integration by parts, we have the bound
$$ \varphi(y) \ll_{K} (1+|y|)^{-K} \|\psi\|_{C^K}$$
for any $K>0$; also, for any $r \geq 0$, we can write $\varphi$ as an $r^{\mathrm{th}}$ derivative $\varphi = \varphi^{(r)}_r$ of another function $\varphi_r(y)$ that also obeys the bounds.
\begin{equation}\label{varphy}
  \varphi_r(y) \ll_{K,r} (1+|y|)^{-K} \|\psi\|_{C^K}
\end{equation}
for any $K>0$ and $r \geq 0$.

By the triangle inequality, it thus suffices to show that
$$ \E \varphi\left( m \left(\sum_{i=1}^j \theta_i(a_i-a_{\sigma(i)}) + \xi\right)\right) \ll_{A,j} m^{-A} \|\psi\|_{C^{K_{A,j}}}$$
for any $A > 0$ and any permutation $\sigma$.  To estimate this expression, we perform a partition of unity
$$ 1 = \eta_-(\theta_i) + \eta_0(\theta_i) + \eta_+(\theta_i),$$
where $\eta_\pm$ is a smooth bounded function supported on $[\pm \frac{1}{2}-m^{-0.1}, \pm \frac{1}{2}+m^{0.1}]$ for either choice of sign $\pm$, and $\eta_0$ is a smooth function supported on $[-\frac{1}{2} + m^{0.1}/2, \frac{1}{2}-m^{0.1}/2]$ that obeys derivative estimates $\eta_0^{(r)}(\theta) \ll_r m^{0.1r}$ for all $r \geq 0$.  We can then partition the above expectation into $O(1)$ components of the form
\begin{equation}\label{ephim}
   \E \varphi\left( m \left(\sum_{i=1}^j \theta_i(a_i-a_{\sigma(i)}) + \xi\right) \right) \prod_{i=1}^j \eta_{\epsilon_i}(\theta_i)
\end{equation}
where $\epsilon_i \in \{-,0,+\}$.

Fix the $\epsilon_i$. Suppose that there is some sign $i$ for which $\epsilon_i = 0$, and we condition $a_i - a_{\sigma(i)}$ to be larger than $m^{-0.1}$ in magnitude.  Then by writing $\varphi = \varphi^{(r)}_r$ and repeatedly integrating by parts in $\theta_i$ using \eqref{varphy} (noting that no boundary terms appear as $\eta_{\epsilon_i}$ is supported in the interior of $[0,1]$), one can see that the contribution of this term is acceptable (every integration by parts gains a factor of $\gg m^{0.9}$ while losing a factor of $\ll m^{0.1}$).  Thus we may condition to the event that $a_i - a_{\sigma(i)} = h_i$ for all $i$ with $\epsilon_i=0$, where the $h_i$ are real numbers with $|h_i| \leq m^{-0.1}$.  Note that if $\epsilon_i=0$ for all $i$ in a cycle of $\sigma$, then the $h_i$ are constrained to sum to zero along that cycle, but otherwise there is no constraint on the $h_i$.  Once one fixes the $h_i$, there will be some subset $(a_i)_{i \in I}$ of the $a_1,\dots,a_j$ that can still vary independently in an interval which lies within $O(m^{-0.1})$ of $[0,1]$, and which determine the remainder of the $a_i$.  This is best illustrated with an example.  Suppose for instance that $j=6$ and the permutation $\sigma$ is given by
$$ \sigma(1)=3; \quad \sigma(3)=2; \sigma(2)=1; \sigma(4)=5; \sigma(5)=6; \sigma(6)=5$$
and suppose that $\epsilon_i=0$ for $i=1,3,4$, thus constraining the quantities
$$ a_1 - a_3 = h_1; a_3 - a_2 = h_3; a_4 - a_5 = h_4.$$
Then the variables $a_1, a_4, a_6$ (say) can still vary independently, subject to the constraints
\begin{align*}
 -1/2, -1/2+h_1, -1/2+h_1+h_3 \leq a_1 &\leq 1/2, 1/2+h_1, 1/2+h_1+h_3;\\ \-1/2, -1/2+h_4 \leq a_2 &\leq 1/2+h_4;\\
 -1/2 \leq a_6 &\leq 1/2
 \end{align*}
and with the other values of $a_i$ determined by the relations
$$ a_3 = a_1-h_1; a_2 = a_1 - h_1 - h_3; a_5 = a_4 - h_4.$$
So we can take $I = \{1,4,6\}$ in this case (though other choices are also available).  

If we fix the $h_i$ for $\epsilon_i = 0$, as well as all of the $\theta_i$, then the expression $\sum_{i=1}^j \theta_i(a_i-a_{\sigma(i)})$ can then be written as an affine combination
$$ \sum_{i \in I} c_i a_i + c$$
of the free variables $a_i$, where the coefficients $c_i$ are all within $O(m^{-0.1})$ of a half-integer, and the shift $c$ is also of size $O(m^{-0.1})$.  For instance, continuing the previous example, we have the expression
$$ \theta_1(a_1-a_3) + \theta_3(a_3-a_2) + \theta_2(a_2-a_1) + \theta_4(a_4-a_5) + \theta_5(a_5-a_6) + \theta_6(a_6-a_4)$$
which simplifies to
$$ c_1 a_1 + c_4 a_4 + c_6 a_6 + c$$
where $c_1 = 0$, $c_4 = \theta_5 - \theta_6$, $c_6 = \theta_6 - \theta_5$, and $c = \theta_1 h_1 + \theta_3 h_3 + \theta_2 (-h_1-h_3) + \theta_4 h_4 + \theta_5 (-h_4)$.  Note that as $\theta_5, \theta_6$ are constrained to lie within $O(m^{-0.1})$ of either $1/2$ or $-1/2$, that $c_1,c_4,c_6$ lie within $O(m^{-0.1})$ of a half-integer, and because $h_1,h_3,h_4 = O(m^{-0.1})$, we also have $c = O(m^{-0.1})$.

We consider $c, c_i$ to be deterministic coefficients and just study the effect of averaging in the $a_i$.  We condition to be fixed any $a_i$ for which $c_i$ lies within $O(m^{-0.1})$ of the origin, leaving a modified expression $\sum_{i \in I'} c_i a_i + c'$ for some subset $I'$ of $I$ with $c_i$ within $O(m^{-0.1})$ of a non-zero half-integer for each $i \in I$, and some $c' = O(m^{-0.1})$.  Each $c_i a_i$ is uniform in an interval of the form $[O(m^{-0.1}), c_i + O(m^{-0.1})]$ if $c_i$ is positive, or $[c_i + O(m^{-0.1}), O(m^{-0.1})]$.  After repeatedly convolving, we conclude that $\sum_{i \in I'} c_i a_i + c'$ has a distribution that is piecewise polynomial; indeed, the $|I'|^{\mathrm{th}}$ (distributional) derivative of this distribution is a linear combination of Dirac delta functions located within $O(m^{-0.1})$ of half-integers, with coefficients of $O(1)$.  Meanwhile, by \eqref{low}, $\xi$ lies at distance $\geq \frac{1}{\log^A(2+m)}$ from any half-integer.  If we write $\varphi = \varphi^{(r)}_r$ for $r = |I'|$ and repeatedly integrate by parts using \eqref{varphy}, we thus see that this contribution to \eqref{ephim} is acceptable.

\subsection{The low-frequency estimate}\label{lowfreq-sec}

Now we prove the low-frequency estimate \eqref{low-freq}. Again, the first step is to transition from a fixed index estimate to a fixed energy estimate.  Let $\alpha = O(1)$ be a (deterministic) quantity depending on $i,N$ to be chosen later.  By \eqref{imi-2} and the triangle inequality it will suffice to show that
$$ \E \left|\sum_{l=1}^m \tilde g_{i+l} - \alpha g_{i+l}\right|^2 \ll m^{2-c} + o(1).$$
From Theorem \ref{interlacing-univ} and a standard truncation argument (using Theorem \ref{gap-bounds}(ii)) we have
$$  \E \left|\sum_{l=1}^m \tilde g_{i+l} - \alpha g_{i+l}\right|^2 =  \E \left|\sum_{l=1}^m \tilde g_{j+l} - \alpha g_{j+l}\right|^2 + o(1)$$
whenever $j = i + o(n)$.  Thus we may replace the left-hand side of \eqref{low-freq} by the averaged quantity
\begin{equation}\label{mag}
  \frac{1}{\log^{10} N} \sum_{i \leq j < i + \log^{10} N} \E \left|\sum_{l=1}^m \tilde g_{j+l} - \alpha g_{j+l}\right|^2
\end{equation}
up to negligible error.

Define the random function $F: \R \to \R$ by setting $F(t) = 1-\alpha$ if $f_{\gamma_{i/N}}(t)$ lies in an interval $[\lambda_j, \lambda'_j]$ for some $1 \leq j < N$, and $-\alpha$ otherwise.  Observe that the quantity $\sum_{l=1}^m \tilde g_{j+l} - \alpha g_{j+l}$ is equal to $\int_{f_{\gamma_{i/N}}^{-1}(\lambda_{j+1})}^{f_{\gamma_{i/N}}^{-1}(\lambda_{j+m+1})} F(t)\ dt$ up to negligible errors.  This in turn is equal to
$$ \int_{f_{\gamma_{i/N}}^{-1}(\lambda_{j+1})}^{f_{\gamma_{i/N}}^{-1}(\lambda_{j+1}+m)} F(t)\ dt + O( |g_{j+1}+\dots+g_{j+m}-m|).$$
The contribution of the error is acceptable thanks to Theorem \ref{gap-bounds}(iv), so we reduce to showing that
$$ \E \sum_{i \leq j < i + \log^{10} N} \left|\int_{f_{\gamma_{i/N}}^{-1}(\lambda_{j+1})}^{f_{\gamma_{i/N}}^{-1}(\lambda_{j+1}+m)} F(t)\ dt\right|^2 \ll (m^{2-c}+o(1)) \log^{10} N.$$
Using \eqref{rigidity}, the expression \eqref{mag} expands as
$$ \E \int_0^{\log^{10} N} \int_0^{\log^{10} N} 1_{|s-t| \leq m} a(s,t) F(t) F(s)\ ds dt$$
up to acceptable errors, where $a(s,t)$ is the number of $j$ such that
$$ f_{\gamma_{i/N}}(s)-m, f_{\gamma_{i/N}}(t)-m \leq \lambda_{j+1} \leq f_{\gamma_{i/N}}(s), f_{\gamma_{i/N}}(t).$$
Using \eqref{varfx}, we see that $a(s,t)$ has mean $m - |s-t|+o(1)$ and variance $O(\log(2+m)+o(1))$.  The contributions of the fluctuations are acceptable by Cauchy--Schwarz, so by the triangle inequality it suffices to show that
$$ \E \int_0^{\log^{10} N} \int_0^{\log^{10} N} (m-|s-t|)_+ F(t) F(s)\ ds dt \ll (m^{2-c}+o(1)) \log^{10} N.$$
Up to negligible errors, the left-hand side is
$$ \int_0^{\log^{10} N} \E \left|\int_t^{t+m} F(s)\ ds\right|^2\ dt,$$
so by the triangle inequality it suffices to show that
$$ \E \left|\int_t^{t+m} F(s)\ ds\right|^2 \ll m^{2-c} + o(1)$$
for each $0 \leq t \leq \log^{10} N$.  For sake of notation we shall just establish this for $t=0$ (the general case is similar, and basically amounts to shifting $i$ by $t$).

At this point we shall finally need the determinantal structure of the minor process.
Recall the joint eigenvalue process $\tilde \Sigma$ defined in \eqref{td-def}.  Observe that
$$ F(t) = \# (\tilde \Sigma \cap \{N\} \times (-\infty,f_x(t))) - \# (\tilde \Sigma \cap \{N-1\} \times (-\infty,f_x(t))) - \alpha$$
for almost every $t$, hence
\begin{equation}\label{fs}
  \int_0^m F(s)\ ds = m \left(\sum_{(N,\lambda) \in \tilde \Sigma} \eta(\lambda) - \sum_{(N-1,\lambda') \in \tilde \Sigma} \eta(\lambda') - \alpha\right)
\end{equation}
where $\eta$ is the continuous piecewise linear cutoff function
$$ \eta(x) \coloneqq \left( 1 - \sqrt{N/2} \rhosc(\gamma_{i/N}) (x - \sqrt{2N} \gamma_{i/N})_+ / m \right)_+.$$
This is a linear statistic in the minor process $\tilde \Sigma$.  A determinantal kernel
$$ \tilde K \colon (\{N-1,N\} \times \R) \times (\{N-1,N\} \times \R)  \to \R$$
of this process was computed in \cite{johansson-nordenstam} (see also \cite{adler} for an equivalent kernel).  If we abbreviate
$$ \tilde K_{N_1,N_2}(x_1,x_2) \coloneqq \tilde K((N_1,x_1), (N_2,x_2)),$$
then we can describe the kernel in terms of the orthonormal eigenfunctions
$$ \phi_n(x) \coloneqq h_n(x) e^{-x^2/2}$$
of the harmonic oscillator $-\frac{d^2}{dx^2}+x^2$ by the formulae
\begin{align*}
  \tilde K_{N,N}(x,x') &= \sum_{j=0}^{N-1} \phi_j(x) \phi_j(x') \\
  \tilde K_{N-1,N-1}(y,y') &=\sum_{j=0}^{N-2} \phi_j(y) \phi_j(y') \\
  \tilde K_{N-1,N}(y,x') &= \sum_{j=1}^{N-1} \sqrt{j} \phi_{j-1}(y) \phi_j(x') \\
  \tilde K_{N,N-1}(x,y') &= - \sum_{j=N}^\infty \frac{1}{\sqrt{j}} \phi_j(x) \phi_{j-1}(y'),
\end{align*}
where the infinite sum for the last kernel converges pointwise a.e. and in $L^2$; (see \cite[Theorem 1.3, (5.6), (5.7), Lemma 5.6]{johansson-nordenstam}).  We can then easily compute the first moment:
\begin{align*}
  \E \int_0^m F(s)\ ds &= \int_\R m \eta(x) (\tilde K_{N,N}(x,x) - \tilde K_{N-1,N-1}(x,x))\ dx - \alpha m\\
&= m (\int_\R \phi_{N-1}(x)^2 \eta(x)\ dx - \alpha).
\end{align*}
If we then set
$$ \alpha \coloneqq \int_{-\infty}^{\gamma_{i/N}} \phi_{N-1}(x)^2\ dx$$
then the mean of $\int_0^m F(s)\ ds$ is $o(1)$.  It now suffices to obtain the variance bound
$$ \Var \frac{1}{m} \int_0^m F(s) \ll m^{-c} + o(1).$$
Using \eqref{fs} and standard determinantal process calculations, the left-hand side may be expanded as
\begin{align*}
  &\int_\R \eta(x)^2 K_{N,N}(x,x)\ dx + \int_\R \eta(y)^2 K_{N-1,N-1}(y,y)\ dy\\
  &- \int_\R \int_\R \eta(x) \eta(x') K_{N,N}(x,x') K_{N,N}(x',x)\ dx dx'\\
  &- \int_\R \int_\R \eta(y) \eta(y') K_{N-1,N-1}(y,y') K_{N-1,N-1}(y',y)\ dy dy'\\
  &- 2 \int_\R \int_\R \eta(x) \eta(y') K_{N,N-1}(x,y') K_{N-1,N}(y',x)\ dx dy'.
\end{align*}
Expanding out into the orthonormal basis $\phi_n$, and introducing the coefficients
$$ a_{j,k} \coloneqq \int_\R \eta \phi_j \phi_k,$$ 
we have
$$ \int_\R \eta(x)^2 K_{N,N}(x,x)\ dx = \sum_{j < N} \|\eta \phi_j\|_{L^2}^2 = \sum_{j<N} \sum_k |a_{j,k}|^2$$
and
$$\int_\R \int_\R \eta(x) \eta(x') K_{N,N}(x,x') K_{N,N}(x',x)\ dx dx' = \sum_{j,k < N} |a_{j,k}|^2$$
and hence
\begin{align*}
&\int_\R \eta(x)^2 K_{N,N}(x,x)\ dx\\
&\quad  - \int_\R \int_\R \eta(x) \eta(x') K_{N,N}(x,x') K_{N,N}(x',x)\ dx dx' \\
&\quad =
\sum_{j < N \leq k} |a_{j,k}|^2.
\end{align*}
Similarly
\begin{align*}
  &\int_\R \eta(x)^2 K_{N-1,N-1}(y,y)\ dy  \\
  &\quad - \int_\R \int_\R \eta(y) \eta(y') K_{N-1,N-1}(y,y') K_{N-1,N-1}(y',y)\ dy dy' \\
  &\quad =
\sum_{j < N \leq k} |a_{j,k}|^2.
\end{align*}
Also we have
$$\int_\R \int_\R \eta(x) \eta(y') K_{N,N-1}(x,y') K_{N-1,N}(y',x)\ dx dy' = \sum_{1 \leq j < N \leq k} \frac{\sqrt{j}}{\sqrt{k}} a_{j,k} a_{j-1,k-1}.$$
Thus the variance can be expressed as
\begin{equation}\label{div}
\sum_{j < N \leq k} a_{j,k}^2 + \sum_{j < N-1 \leq k} a_{j,k}^2 - \sum_{1 \leq j < N \leq k} \frac{\sqrt{j}}{\sqrt{k}} a_{j,k} a_{j-1,k-1}.
\end{equation}
To estimate the coefficients $a_{j,k}$, we recall the classical raising and lowering identities
\begin{equation}\label{raising}
   -\phi'_n + x \phi_n = \sqrt{2(n+1)} \phi_{n+1}
\end{equation}
and
\begin{equation}\label{lowering}
   \phi'_n + x \phi_n = \sqrt{2n} \phi_{n-1}
\end{equation}
(with the convention that $\phi_{-1}=0$), leading to the eigenfunction identity

$$-\phi''_n + x^2 \phi_n = (2n+1) \phi_n.$$
This in turn implies the Wronskian identity
\begin{equation}\label{2jk}
   2(j-k) \phi_j \phi_k = (\phi'_k \phi_j - \phi'_j \phi_k)'.
\end{equation}
Integrating this against $\eta$, we conclude that
$$ a_{j,k} = \frac{1}{2(j-k)} \int_\R \eta' (\phi_j \phi'_k - \phi'_j \phi_k)$$
for $j \neq k$.  Since $\eta'$ is supported in a $O(mN^{-1/2})$ neighborhood of $\sqrt{2N} \gamma_{i/N}$ and has a total mass of $1$, standard Plancherel--Rotach asymptotics then give the bound
\begin{equation}\label{jk}
   a_{j,k} \ll \frac{1}{|j-k|}
\end{equation}
when $1 \leq j < N \leq k$ (or $1 \leq j < N-1 \leq k$).  This bound does not adequately control \eqref{div} by itself (there is a logarithmic divergence), but at least can handle the $j=1$ contribution of the first sum of \eqref{div}, leaving one with the task of establishing the bound
$$ \sum_{1 < j < N \leq k} a_{j,k}^2 + a_{j-1,k-1}^2 - \frac{\sqrt{j}}{\sqrt{k}} a_{j,k} a_{j-1,k-1} \ll m^{-c} + o(1).$$
We can write the left-hand side as
\begin{equation}\label{ajp}
  \sum_{1 < j < N \leq k} (a_{j,k}-a_{j-1,k-1})^2 + \left(1-\frac{\sqrt{j}}{\sqrt{k}}\right) a_{j,k} a_{j-1,k-1}.
\end{equation}
To improve the bound \eqref{jk}, we use the lowering identity \eqref{lowering} to write
$$ \phi_j \phi'_k - \phi'_k \phi_j = \sqrt{2k} \phi_j \phi_{k-1} - \sqrt{2j} \phi_{j-1} \phi_k$$
and hence by several applications of \eqref{2jk} and integration by parts we have
$$ a_{j,k} = \frac{\sqrt{2k}}{2(j-k)(j-k+1)} \int_\R \eta'' (\phi_j \phi'_{k-1} - \phi'_j \phi_{k-1})
- \frac{\sqrt{2j}}{2(j-k)(j-k-1)} \int_\R \eta'' (\phi_{j-1} \phi'_{k} - \phi'_{j-1} \phi_{k})$$
assuming $1 \leq j < N\leq k$ and $k-j > 1$.  As $\eta''$ is supported in a $O(mN^{-1/2})$ neighborhood of $\sqrt{2N} \gamma_{i/N}$ and has a total mass of $N^{1/2}/m$, Plancherel--Rotach asymptotics give the bound
$$ a_{j,k} \ll \frac{\sqrt{kN}}{m |j-k|^2}$$
in this case.  A standard calculation then shows that the contribution of those $j,k$ with $|j-k| \geq N/m^{1/2}$ is acceptable, so we can restrict attention to the case $|j-k| < N/m^{1/2}$.  From Taylor expansion we then have $1-\frac{\sqrt{j}}{\sqrt{k}} \ll \frac{j-k}{N}$, so the contribution of the second term of \eqref{ajp} is acceptable just from \eqref{jk}.  It thus remains to show that
$$   \sum_{1 \leq j < N \leq k: |j-k| < N/m^{1/2}} (a_{j,k}-a_{j-1,k-1})^2 \ll m^{-c} + o(1).$$
From \eqref{lowering}, \eqref{raising} and the Leibniz rule one has
$$ (\phi_j \phi_{k-1})' = \sqrt{2j} \phi_{j-1} \phi_{k-1} - \sqrt{2k} \phi_j \phi_k;$$
integrating this against $\eta$, and then integrating by parts we conclude that
$$ \int_\R \eta' \phi_j \phi_{k-1} = - \sqrt{2j} a_{j-1,k-1} + \sqrt{2k} a_{j,k}.$$
By Plancherel--Rotach asymptotics we have
$$ \int_\R \eta' \phi_j \phi_{k-1} \ll N^{-1/2};$$
from this estimate, the bound $\sqrt{2j} = (1 + O(|j-k|/N)) \sqrt{2k}$, and \eqref{jk}, we conclude that
$$ a_{j,k} = a_{j-1,k-1} + O( 1/N )$$
from which the claim follows.

\bibliographystyle{alpha}

\end{document}